\documentclass{article}

\usepackage{graphicx}

\usepackage[fleqn]{amsmath}
\usepackage{amsthm}
\usepackage{amssymb}
\usepackage{amsbsy} 
\usepackage{amsfonts}
\usepackage{mathrsfs}
\usepackage[all]{xy}
\usepackage{amstext}
\usepackage{amscd}
\usepackage[all]{xy}
\usepackage[dvips]{epsfig}
\usepackage{psfrag}
\usepackage{enumerate}
\usepackage{flafter}
\allowdisplaybreaks

\theoremstyle{plain}
\newtheorem{thm}{Theorem}[section]
\newtheorem{prop}[thm]{Proposition}
\newtheorem{cor}[thm]{Corollary}
\newtheorem{lem}[thm]{Lemma}
\theoremstyle{definition}
\newtheorem{exa}[thm]{Example}

\newtheorem{rem}[thm]{Remark}
\newtheorem{defn}[thm]{Definition}
\newtheorem{prob}[thm]{Problem}

\def\det{\mathop{\mathrm{det}}\nolimits}
\def\Im{\mathop{\mathrm{Im}}\nolimits}
\def\Ker{\mathop{\mathrm{Ker}}\nolimits}
\def\Coker{\mathop{\mathrm{Coker}}\nolimits}
\def\Hom{\mathop{\mathrm{Hom}}\nolimits}
\def\Ext{\mathop{\mathrm{Ext}}\nolimits}
\def\Tor{\mathop{\mathrm{Tor}}\nolimits}
\def\End{\mathop{\mathrm{End}}\nolimits}

\newcommand{\tri}{{ \lhd}}

\newcommand{\lra}{\longrightarrow}
\newcommand{\ra}{\rightarrow}
\newcommand{\Q}{{\Bbb Q}}
\newcommand{\R}{{\Bbb R}}
\newcommand{\Z}{{\Bbb Z}}
\newcommand{\As}{{\rm As }}
\newcommand{\D}{{\cal D}_{g}}

\newcommand{\DD}{{{\cal D}^{\rm ns}_{g}}}

\newcommand{\M}{{\cal M}}
\newcommand{\mm}{\overline{\bf{m} }}

\newcommand{\pc}[2]{\mbox{$\begin{array}{c}
\includegraphics[scale=#2]{#1.eps}
\end{array}$}}
\begin{document}
\large
\begin{center}
{\bf\Large Bilinear-form invariants of Lefschetz fibrations over the 2-sphere}
\end{center}
\vskip 1.5pc
\begin{center}
{\Large Takefumi Nosaka}\end{center}\vskip 1pc\begin{abstract}\baselineskip=13pt \noindent
We introduce invariants of, in general, Hurwitz equivalence classes with respect to arbitrary group $G$.
The invariants are constructed from any right $G$-modules $M$ and any $G$-invariant bilinear function on $M$, and are of bilinear forms.
For instance, when $G$ is the mapping class group of the closed surface, $\M_g$, we get an invariant of 4-dimensional Lefschetz fibrations over the 2-sphere.
Moreover, the construction is applicable for the quantum representations of $\M_g $
derived from Chern-Simons field theory.
We also see that our invariant is unstable with respect to fiber sum of Lefschetz fibrations.

\end{abstract}

\begin{center}
\ \ \ mapping class group, monodromy, link \ \ \
\end{center}
\large
\baselineskip=16pt
\section{Introduction}
A (genus-$g$) Lefschetz fibration \cite{Kas,Matsumoto} is 
a smooth map $\pi : E \ra S^2$ from a closed smooth 4-manifold $E$
that is a $\Sigma_g$-fiber bundle projection away from finitely many singular points.
Here, $\Sigma_g $ is the closed surface of genus $g$, and the map $\pi$ near the singular points is required to appear in appropriate oriented local complex coordinates as $\pi(z_1,z_2)=z_1z_2.$
Lefschetz fibrations provide a topological method available for studying 4-dimensional geometry,
as in complex surfaces, symplectic structures, and Stain surfaces.
Further,
a careful observation of tubular neighborhoods around the singular fibers
develops the close relation to
3-dimensional contact geometry and the mapping class group, $\M_g$, of $\Sigma_g $; see an explanatory book \cite{OS} for details. 


We now explain an approach to Lefschetz fibrations, as a general setting from Hurwitz equivalence problem.
Let $G$ be a group with identity $1_G$, and let a subset $Z \subset G$ be closed under conjugation. For $m \in \mathbb{N}$,
consider the quotient set, $ {\rm Hur}^m(Z) $, of the set
$$\{ (z_1, \dots, z_m ) \in Z^m | \ z_1 \cdots z_m =1_G\ \}$$
modulo the following relations:
\begin{equation}\label{Hureq1} (z_1, z_2, \dots, z_m ) \sim ( z^{-1}z_1 z, z^{-1}z_2 z, \dots, z^{-1} z_m z ) , \end{equation}
\begin{equation}\label{Hureq2} (z_1, z_2, \dots, z_m ) \sim (z_1, \dots, z_{i-1} ,z_{i+1}, z_{i+1}^{-1}z_i z_{i+1} , z_{i+2}, \dots, z_m ) , \end{equation}
for any $1 \leq i < m$ and $z \in Z$.
An element of this set ${\rm Hur}^m(Z)$ is called a {\it Hurwitz equivalence class}, and
is roughly viewed as a monodromy characterized by $G$
over the 2-sphere $S^2$ with $m$-points removed.
This viewpoint
emphasised the importance of studying $ {\rm Hur}^m(Z) $;
As examples in topology, in the case where $G$ is the mapping class group $\M_g $ (resp. the braid group $B_n$),
the associated set $ {\rm Hur}^m(Z) $ is bijective to the set of isomorphism classes of Lefschetz fibrations (resp. of simple surface-braids); see, e.g., \cite{Kas,Matsumoto,PY,Kam}, Examples \ref{LFrei} and \ref{surfacebraidrei}.
Nevertheless, for any group $G$, there are few invariants with respect to the Hurwitz equivalence, although the definition is seemingly simple.
For instance, even concerning Lefschetz fibrations related to symplectic geometry,
most of useful approaches are (geometric) characteristic classes,
group-theoretic study for the monodromy; see, e.g, \cite{EK,EN,FS,Ozb,Sato,PY}
(Refer to a recent work \cite{EK} from ``chart diagrams").

In this paper, for such any pair $Z \subset G$,
we generally introduce an algorithm to provide invariants of Hurwitz equivalence classes $ {\rm Hur}^m(Z) $.
The idea follows from the Hopf fibration $ S^3 \ra S^2 $.
Since the preimage of $m$-points on $S^2$ is the $(m,m)$-torus link $T_{m,m}$ in $S^3$,
our invariants are roughly defined as link-invariants of the link $T_{m,m}$.
To be precise, we show that (Lemma \ref{aa53}) any invariant of representations of link groups
yields an invariant of $ {\rm Hur}^m(Z) $.
Thus, for concrete computations and applications of such invariants,
this paper mainly employs the link-invariant $\mathcal{Q}_{\psi}$ which was defined from cup products 
(see \S \ref{ss211} for the reason why this paper does not use other standard link-invariants).
The interesting point is that this  $\mathcal{Q}_{\psi}$ is elementally constructed from any
$G$-module $M$ and any $G$-invariant bilinear function $\psi$, and is valued as a bilinear form.
In conclusion, we obtain bilinear forms from $ {\rm Hur}^m(Z) $, 
as mentioned in the abstract above.




For their applications,
this paper mainly focuses on Lefschetz fibrations over $S^2$.
As mentioned above,
we shall let $G$ be the mapping class group $\mathcal{M}_g$.
Then,
by virtue of bilinear form theory (see \cite{HM}) and the help of computer,
we can compute the invariants of many Lefschetz fibrations of low genus fibers
(see \S \ref{ssg2}). 
Furthermore, in \S \ref{ss2119}-\ref{ssg2},
we study properties of the bilinear forms, and obtain some interesting results as follows.

As the simplest case, in \S \ref{ssg2}, we study the resulting bilinear form $\mathcal{Q}_{\psi}$
obtained from the standard symplectic representation of $\M_g$, that is, we set $M:= H^1(\Sigma_g;A )$ and the symplectic form $\psi$ on $M$.
We then analyze the rational part and torsion part in turn.
First, in rational case $A=\Z$, one will show the unimodularity of the form $\mathcal{Q}_{\psi}$ (Theorem \ref{aa333c}), and show that
the signature of $\mathcal{Q}_{\psi}$ is essentially equal to the signature of the total space of the Lefschetz fibration.
As a result, we give a simpler formula for computing 4-dimensional signature (This formula is a modification of Meyer 2-cocycle \cite{Meyer};
see Appendix \ref{asss303k1} for details).
On the other hand, if $A$ is the torsion group $\Z/P$ for some prime $P \in \Z$,
the resulting form $\mathcal{Q}_{\psi}$ is not topological, but capture some group theoretic fiber-structures of Lefschetz fibrations.
A typical evidence is that the form $\mathcal{Q}_{\psi}$ seems something unstable:
Precisely, it does not always have the additivity with respect to (twist) fiber sums of Lefschetz fibrations;
seen Table \ref{G3} and Propotision \ref{vbbw}.
Since the sum is a powerful tool for constructing new symplectic 4-manifolds as in \cite{FS,PY}),
we hope a usefulness of the torsion invariant. 

\

Moreover, we emphasize that the quantum representation is applicable to the setting of our invariants. 
Witten \cite{Wi} had made a prophetic discovery that
the Chern-Simons quantum field theory of level $k$ 
produces 3-manifold invariants and representations of $\mathcal{M}_g $;
afterward the prophecy in some cases are mathematically formulated from several branches,
e.g., conformal field theory.
Here we remark that the formulation has a difficulty that, by an obstacle of ``$p_1$-structures (or 2-framing anomaly"; see \cite{BHMV,Wi}),
the resulting representation is a right module not of $\mathcal{M}_g$,
but of a $\Z$-central extension of $\mathcal{M}_g$. 
However, starting from the study of such extensions (see \cite{Nos1} for its finite presentation),
we will apply the quantum representation to our invariants.
Furthermore, the associated bilinear forms of Lefschetz fibrations are computable in some cases.
For instance, in \S \ref{s3s331wlw},
we compute the associated invariants with $SU(2)$-gauge of level 2 (i.e., a quantization of
the spin structure on $\Sigma_g$); see Table \ref{G3} for the computation.
As a result, 
we can
detect some pairs of non-isomorphic Lefschetz fibrations (Proposition \ref{vbbw});
however, 
unfortunately, the invariants seem not so strong, in contrast to the Jones knot-polynomial.


With respect to other gauges of higher level, the bilinear forms $\mathcal{Q}_{\psi}$ arising from the quantum representations are entirely mysterious.
However, compared with a TQFT-like physical interpretation as in the Donaldson invariant,
it may be worth for the futures to hope
that the philosophy of Chern-Simons perturbative \cite{Wi} would partially produce invariants of 4-manifolds with somewhat structures,
and that Lefschetz fibrations would be good models which are affected by few 3-dimensional quantum obstructions.






\

\noindent
{\bf Conventional notation}.
Throughout
this paper, we write $G$ for a group, and do $Z$ for a subset of $G$ as above.
Furthermore {\it an $m$-tuple (of $Z$)} is $(z_1, \dots, z_m ) \in Z^m$ with $z_1\cdots z_m=1$,
and is denoted by $\mathbf{z}$.
We denote by $\Sigma_{g,r}$ the oriented closed surface of genus $g \geq 2$ with $r$-boundaries.
By $A$ we mean a commutative ring with identity and with involution $\bar{}:A \ra A$.


\section{From Hopf fibration to Hurwitz equivalence classes}\label{ss211}
We will describe explicitly an algorithm to get something invariant with respect to the Hurwitz equivalence relation; see the introduction for the definition of ${\rm Hur}^m(Z) $.

For this purpose, we start with a short review of the Hopf fibration $\mu:S^3 \ra S^2$ with fiber $S^1$. The fibration is
formulated by the restriction on $S^3$ of the following map:
$$ \mu: \mathbb{C}^2 \lra \mathbb{C} \times \R, \ \ \ (z,w) \longmapsto ( 2 z \bar{w}, \ |z|^2-|w|^2). $$
Then we can easily see, by definitions, that the preimage of $m$-points, $\{ b_1, \dots, b_m\} \subset S^2$, is the $(m,m)$-torus link $T_{m,m}$.
Considering the Euler class, we can find a trialization of the $S^1$-bundle restricted on the complementary space $S^2 \setminus \{ b_1, \dots, b_m \} $,
that is, a homeomorphism
$$S^3 \setminus T_{m,m} \cong S^1 \times (S^2 \setminus \{ b_1, \dots, b_m \} ).$$
In particular, $ \pi_1(S^3 \setminus T_{m,m}) \cong \Z \times F_{m-1}$, where $F_{m-1}$ is the free group of rank $m-1$. Hence
\begin{prop}\label{aa51433}
The complementary space $S^3 \setminus T_{m,m}$ is an Eilenberg MacLane space of $\Z \times F_{m-1}$.
\end{prop}
\noindent
Moreover, the Wirtinger presentation yields a presentation of $\pi_1(S^3 \setminus T_{m,m})$ as follows:
\begin{equation}\label{aa13}
\langle \ a_1, \dots, a_m \ | \ a_1 \cdots a_m = a_2 \cdots a_m a_1 = \cdots = a_m a_1 a_2 \cdots a_{m-1} \ \rangle .
\end{equation}
Here, $a_i$ corresponds to the meridian associated with the arc $\alpha_i$ in Figure \ref{tg222},
and the product $ a_1 \cdots a_m$ generates the summand $\Z \subset \Z \times F_{m-1}$.

Next, to explain Lemma \ref{aa53} below, we now set up terminologies.
Wirtinger presentation in knot theory implies that every homomorphism from every link group to the group $G$ is formulated
as a map $\{ \textrm{arcs of } D\} \ra G $. Let us consider the set of such maps:
$$ \mathfrak{LHom}_{Z}:= \Bigl\{ \ \gamma : \{ \textrm{arcs of } D\} \ra Z \ \Bigl| \
\begin{array}{c}
\!\!\!\!\!\!\!\!\!\!\! D \ \textrm{is a link diagram of some link }L \subset S^3, \textrm{and} \\
\textrm{this } \gamma \ \textrm{defines a homomorphism } \pi_1(S^3 \setminus L ) \ra G.
\end{array} \Bigr\} .$$
Further, we can equip $\mathfrak{LHom}_{Z}$ with the equivalence relation
by considering Reidemeister moves.

Furthermore, we discuss the link-diagram $D$ of $T_{m,m} $ in the left hand side of Figure \ref{tg222}.
Here, any entry of the linking matrix is $1$, and the band in Figure \ref{tg222} means the $(m-2)$-parallel strands, and
the blackline is the $i$-th strand and the dotted line indicates the $(i+1)$-th strand.
Then, given an $m$-tuple $(z_1, \dots, z_m) \in Z^m$ with $z_1\cdots z_m =1 $,
the assignment $f_{\mathbf{z}}(a_i)=z_i$ defines a homomorphism $ f_{\mathbf{z}}: \pi_1(S^3 \setminus T_{m,m}) \ra G$ according to \eqref{aa13}.
Hence we have a map
$$ \mathcal{Z}: \{ \ \mathbf{z}=(z_1, \dots, z_m) \in Z^m \ | \ z_1\cdots z_m =1 \ \} \lra \mathfrak{LHom}_{Z} ; \ \ \ \ \mathbf{z} \mapsto f_{\mathbf{z}}.$$
\begin{lem}\label{aa53} Let $\mathcal{S} $ be a set, and take a map $ \mathcal{I}: \mathfrak{LHom}_{Z} \ra \mathcal{S} $.
Assume that $ \mathcal{I}$ is invariant with respect to Reidemeister moves of type II and III, and assume that
if two homomorphisms $f_{\mathbf{z}}$ and $ f_{\mathbf{z}}': \pi_1(S^3 \setminus L ) \ra G$ coming from $ \mathfrak{LHom}_{Z} $ are conjugate, $\mathcal{I}(f_{\mathbf{z}})= \mathcal{I}(f_{\mathbf{z}}') $.

Then, the composite $ \mathcal{Z} \circ \mathcal{I} $ induces a map
$ \mathrm{Hur}_m(Z) \ra \mathcal{S}$ by passage to the relations \eqref{Hureq1} \eqref{Hureq2}. 
\end{lem}
\begin{proof}
It suffices to check the invariance with respect to the relations \eqref{Hureq1} \eqref{Hureq2}.
Since the former \eqref{Hureq1} is clear by assumption, we will discuss another \eqref{Hureq2}.
To this end, consider another diagram $D'$ obtained from the above $D$ by exchanging the $i$-th strand for the $(i+1)$-th one (see the right of Figure \ref{tg222});
Notice that
$D$ is related to $D'$ by a finite sequence of Reidemeister moves of type II and III (see Figure \ref{tg222}). 
Therefore, if two $m$-tuples $\mathbf{z}$ and $ \mathbf{z}' $ are related by \eqref{Hureq2},
then the equality $ \mathcal{I} (f_{\mathbf{z}}) = \mathcal{I} ( f_{\mathbf{z}'})$ results from the assumption.
Hence, we complete the proof.
\end{proof}

\vskip -0.5pc
\begin{figure}[h]
$$
\begin{picture}(200,100)
\put(-125,35){\pc{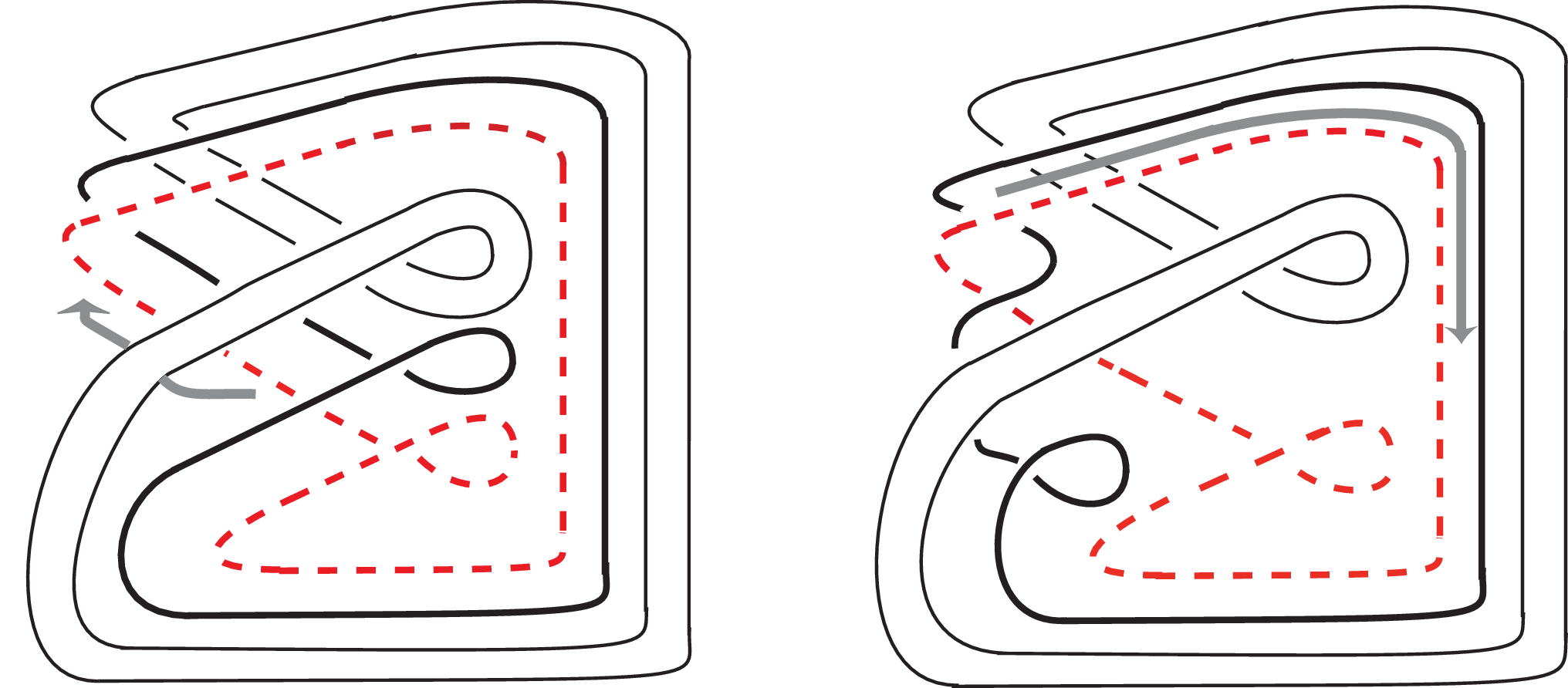}{0.2188125}}

\put(233,35){$\sim $}
\put(104,35){$\sim $}
\put(-16,35){$\sim $}

\put(-136,65){\Large $\alpha_{i+1} $}
\put(-126,52){\Large $\alpha_{i} $}

\put(-23,75){\LARGE $D $}
\put(242,75){\LARGE $D'$}

\put(118,35){\pc{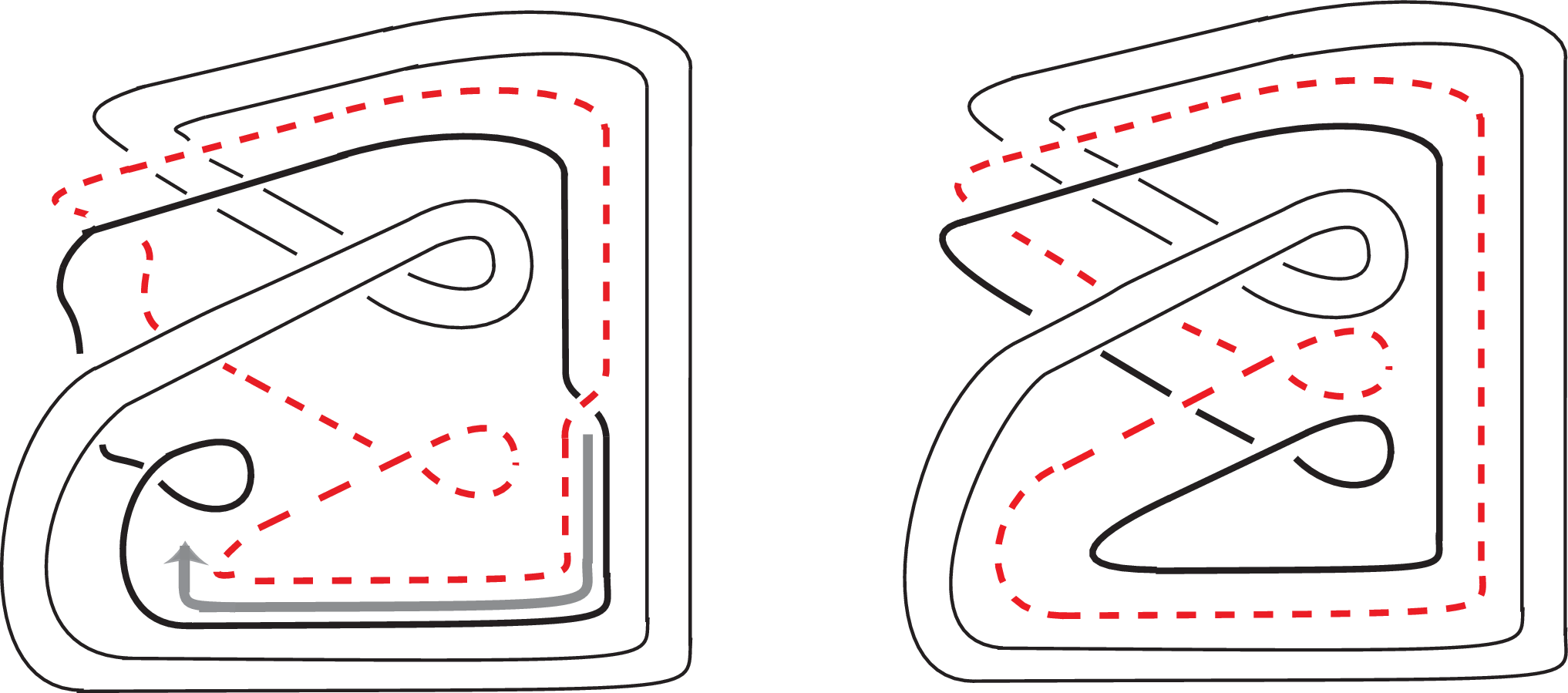}{0.2188125}}


\end{picture}
$$

\

\vskip -0.5pc
\caption{\label{tg222} The exchange between the $i$-th strand and the $(i+1)$-th one of the $(m,m)$-torus link.}
\end{figure}
To summarize, every invariant of representations of link groups
produces that of $ \mathrm{Hur}_m(Z) $.

\begin{rem}\label{aa5312}
Here, we remark such invariants arising from group cohomology, 
as in twisted Alexander polynomials (modules) and the (complex) volume. 
However, Proposition \ref{aa51433} implies that
the invariants factor through the group cohomology of $\Z \times F_{m-1}$ which is entirely simple.
In fact, the author found no non-trivial examples in computer experience;
in most cases, the resulting invariants of $ \mathrm{Hur}_m(Z) $ are expected to be trivial.


To solve such triviality, this paper will employ the link-invariants, which are obtained from cup products with quandle theory and were introduced in \cite{Nos3}.
In fact, we will see (Proposition \ref{isnsdvdef}) that the Hopf fibration gives rise to 
an isomorphism between ``the fundamental quandles" of $S^3\setminus T_{m,m}$ and of $S^2 \setminus \{ b_1, \dots, b_m \}$.
Thus, some approach using quandle theory seem reasonable to accomplish Lemma \ref{aa53}, 
In summary, from the viewpoint of quandle theory, such invariants from $S^2 \setminus \{ b_1, \dots, b_m \}$ are equivalent to those from $S^3\setminus T_{m,m}$.
\end{rem}


\section{Invariants as bilinear forms}\label{ss21}
According to Remark \ref{aa5312} from quandle theory,
Section \ref{ss211} briefly formulates the invariants of linear forms.
After that, in \S\S \ref{ss2119} and \ref{ss21192}, we will state some properties of the bilinear forms.

\subsection{Definition of the invariants as bilinear forms}\label{ss211}
In this subsection, we introduce invariants, which are of bilinear forms up to base change, with respect to the Hurwitz equivalence. 
As explained in Proposition \ref{aa1133c}, these invariants are roughly something like the cohomology pairing from cup products on $D^2$. 

We now review some results in the preceding paper \cite{Nos3}. The author introduced invariants,
which are bilinear forms, with respect to representations $\pi_1(S^3 \setminus L) \ra G$. 
This invariant is generally constructed from any right $\mathrm{As}(Z)$-module $M$ over a ring $A $ and any $\mathrm{As}(Z)$-invariant bilinear function $\psi: M^{\otimes 2} \ra A $.
Here $\mathrm{As}(Z)$ is the abstract group generated by $e_z$ with $z \in Z$ subject to the relation $e_{w^{-1}zw }= e_{w}^{-1}e_z e_w $,
and is called {\it the associated group} \cite{Joy}.
Note that the inclusion $Z \hookrightarrow G$ gives rise to a group homomorphism
\begin{equation}\label{bba} \mathcal{E}: \As(Z) \lra G, \end{equation}
which makes any right $G$-module into a right $\As(Z)$-module. 

We will explicitly formulate the invariants of the torus links as follows.
For an $m$-tuple $\mathbf{z}= (z_1, \dots, z_m ) \in Z^m$ and for $k \leq m $,
we set up an $A$-linear map $\Gamma_{\mathbf{z},k} : M^m \ra M$ defined by
\begin{equation}\label{aac}\Gamma_{\mathbf{z},k} ( x_1, \dots, x_m ):= ( x_{k-1} -x_{k} ) + \sum_{j=k}^{m+k-2 }( x_j -x_{j+1} ) \cdot e_{z_{j+1}}e_{ z_{j+2}} \cdots e_{z_{m+k-1}} \in M.
\end{equation} 
Here the indices are of period $m$, e.g., $x_{m+s}:= x_s \in M $.
Let us denote the kernel $\mathrm{Ker}(\oplus_k \Gamma_{\mathbf{z},k }) $ by $\mathrm{Ker}( \Gamma_{\mathbf{z} }) $.
Here remark that $\mathrm{Ker}( \Gamma_{\mathbf{z} }) $ is always not an $\As(Z)$-submodule, but an $A$-submodule.
\begin{defn}[{\cite[\S 4.3]{Nos3}}]\label{aa11c} 
Let $A, \ M, \ \mathbf{z} \in Z^m$ be as above. 
Let $\psi: M^2 \ra A $ be an $\As(Z)$-invariant $A$-linear function. That is, 
$\psi$ is bilinear over $\Z$ and satisfies the equalities
$$ \psi (a x ,b y )= \bar{a} b \psi (x ,y ), \ \ \ \ \ \ \ \ \ \ \ \psi (x \cdot e_z , y \cdot e_z )= \psi (x ,y ), $$
for any $x,y \in M, \ a,b \in A$ and $z \in Z$.
Then, for $ \ell \in \Z_{>0}$, we consider the $A$-bilinear form $\mathcal{Q}_{\psi, \ell }: \mathrm{Ker}(\Gamma_{\mathbf{z}} )^{\otimes 2} \ra A$
that takes $ (x_1, \dots, x_m) \otimes(y_1, \dots, y_m) $ to 
\begin{equation}\label{bbbdd}
\sum_{k=1}^{m-1 } \psi \bigl( x_{k+\ell -1}-x_{k +\ell }+ \sum_{j=1}^{k -1} (x_{j+\ell -1}-x_{j +\ell })\cdot e_{z_{j+\ell }}e_{ z_{j+\ell +1}} \cdots e_{z_{ k +\ell-1}} ,\ y_{k +\ell} \cdot (1 -e_{z_{k +\ell}}^{-1}) \bigr) \in A . \end{equation}
\end{defn}
It is worth noting that the linear forms are invariant with respect to conjugacy operations \cite[Corollary 2.4]{Nos3}.
Hence, according to Lemma \ref{aa53}, we immediately conclude the following:
\begin{thm}\label{aa4221c}
The correspondence from $ (z_1, \dots, z_m) \in Z^m $ to the bilinear form $ \mathcal{Q}_{\psi, \ell}$ 
up to base change gives rise to an invariant of the Hurwitz equivalence classes.
\end{thm}

\begin{rem}\label{aa21c}
This theorem can be proved by a direct computation in another way.
Actually, according to the discussion of \cite[\S 2]{Nos3}, if another tuple $\mathbf{z}' \in Z^m $ is Hurwitz equivalent to $ \mathbf{z} \in Z^m $,
then we can find an $A$-isomorphism $\mathcal{B}_{\mathbf{z}, \mathbf{z}' }: \mathrm{Ker}(\Gamma_{\mathbf{z}} ) \ra \mathrm{Ker}(\Gamma_{\mathbf{z}'} )$
satisfying the equality $\mathcal{Q}_{\psi, \ell } = \mathcal{Q}_{\psi, \ell }' \circ (\mathcal{B}_{ \mathbf{z}, \mathbf{z}' })^{\otimes 2 } $.
More precisely, regarding the relation \eqref{Hureq1}, we define the map by
$\mathcal{B}_{ \mathbf{z}, \mathbf{z}' } (x_i):= x_i \cdot e_{z}$. Further, concerning another \eqref{Hureq2}, we can verify the equality by setting 
$$ \mathcal{B}_{ \mathbf{z}, \mathbf{z}' } (x_1, \dots, x_m):= (x_1, \dots, x_{i-1} , x_{i+1}, x_{i+1} + (x_{i} -x_{i+1} )\cdot e_{z_{i+1 }}, x_{i+2}, \dots, x_m), $$
although the verifications are slightly complicated.
\end{rem} 
As a consequence, we get invariants of Lefschetz fibrations over $S^2$ and
of surface braids, thanks to the following facts:
\begin{exa}[{Lefschetz fibration over $S^2$}]\label{LFrei}
We now review a well-known theorem in \cite{Kas,Matsumoto}.
To describe this, let $G$ be the mapping class group, $\M_g $, of the closed surface $\Sigma_g$ with $g \geq 2$.
Furthermore, we fix two subsets of $G$ as follows:
\begin{equation}\label{defdg}\D:= \{ \ \tau_{\alpha } \in \M_g \ | \ \alpha \textrm{ is an (unoriented) simple closed curve} \ \gamma \ {\rm in \ } \Sigma_g \ \} , \end{equation}
\begin{equation}\label{defdg11}\DD:= \{ \ \tau_{\alpha } \in \D \ | \ \alpha \textrm{ is a non-separating simple closed curve} \ \gamma \ {\rm in \ } \Sigma_g \ \} , \end{equation}
where the symbol $\tau_{\alpha}$ is the (positive) Dehn twist along $\alpha$.
Let $Z $ be the set $ \D$ in \eqref{defdg}. 
Then, given a Lefschetz fibration,
we can observe that the associated monodromy is interpreted as an $m$-tuple.
Moreover, it follows from \cite{Kas} and \cite[Theorems 2.6 and 2.8]{Matsumoto} that
the interpretation gives a bijection between
the Hurwitz equivalence classes ${\rm Hur}^m(Z)$ and fiber-isomorphism classes of Lefschetz fibrations over $S^2$ with $m$-singular fibers.
\end{exa}
\begin{exa}[{Simple surface braids}]\label{surfacebraidrei}
Let $G$ be the braid group $B_n$,
and let $Z$ be the set of all elements conjugate to either $\sigma_i$ or $ \sigma_i^{-1}$.
Kamada \cite{Kam} showed that the Hurwitz equivalence classes are in 1:1-correspondence with the isomorphism classes of
``simple surface-braids with $m$-branch points of degree $n$"; see \cite{Kam} for the details.
\end{exa}


\begin{rem}\label{akj}
Finally, we comment the previous works on Lefschetz fibration invariants constructed from some quandles.
Zablow \cite[Theorem 5.3]{Zab} combinatorially defined a certain invariant of Lefschetz fibrations over the 2-disk,
which is valued in ``the second quandle homology group" $H_2^Q(\D;\Z)$ or $H_2^Q(\DD;\Z) $ associated with the sets $\D$ and $\DD$.
However the author showed that the homology $H_2^Q(\DD;\Z)$ is $ \Z/2 $ for $g \geq 5$; \cite[\S 4.4]{Nos5}.
Thus, it is sensible to conjecture that, in the case $X= \D$, such invariants might be almost trivial.
Actually, there was no non-trivial example of such invariants with $X= \D$.
In contrast, our invariants constructed from $\M_g$-modules
take non-trivial examples (see \S \ref{ss2132}).
\end{rem}
\subsection{A reduction of the invariants of bilinear forms}\label{ss2119}
In this paper,
we focus on the bilinear forms $\mathcal{Q}_{\psi}$ by virtue of bilinear form theory.
This subsection discusses some reductions to compute the invariants of bilinear forms.

First, we briefly make a reduction of the kernel $\mathrm{Ker}( \Gamma_{\mathbf{z} })$ under an assumption.
\begin{lem}\label{pro2331}
Fix an $m$-tuple $ (z_1, \dots, z_m ) \in Z^m$.
Assume the identity $e_{z_1} \cdots e_{z_m}= c \cdot \mathrm{id}_M $ in $\mathrm{End}(M)$ for some $c \in A $.

If $c=1$, then the kernel $\mathrm{Ker}( \Gamma_{\mathbf{z} }) $=$\mathrm{Ker}(\oplus_k \Gamma_{\mathbf{z},k }) $ is equal to the kernel
$\mathrm{Ker}(\Gamma_{\mathbf{z},k }) $ for any $k$.

If $ (1-c) \cdot \mathrm{id}_M $ is injective, then the kernel $\mathrm{Ker}( \Gamma_{\mathbf{z} }) $ is
the diagonal set $ \mathrm{Diag}(M)\subset M^m$.
\end{lem}
\begin{proof}
By definition, for any $ \mathbf{x}=(x_1, \dots, x_m)\in M^m $, we easily see the equation
$$ \Gamma_{\mathcal{S},k+1} ( \mathbf{x}) = \Gamma_{\mathcal{S},k} ( \mathbf{x}) \cdot e_{z_{m+k}} +(x_{k-1}-x_{k} )\cdot (1 - e_{z_k} \cdots e_{z_{m+k-1 }}) \in M . $$
Hence the assumptions on $ (1-c) \cdot \mathrm{id}_M $ deduce the required conclusions.
\end{proof}
Therefore, to obtain non-trivial examples of $\mathcal{Q}_{\psi}$,
it is reasonable to fix the assumption $e_{z_1} \cdots e_{z_m}= \mathrm{id}_M $.

We now investigate this assumption. 
It is worth noticing that the kernel of the homomorphism $ \mathcal{E}: \As(Z) \ra G $ in \eqref{bba} is contained in the center,
because of the equality
$$ g e_{z} g^{-1} = e_{ \mathcal{E} (g) z \mathcal{E}(g)^{-1}} \in \As(Z) \ \ \ \ \mathrm{for \ any \ } z \in Z, \ g\in \As(Z) .$$
Thus, the product $e_{z_1} \cdots e_{z_m} \in \As(Z) $ lies in the center, since $z_1 \cdots z_m =1_G$.
In summary, the assumption is detected by the image of the center in $\mathrm{End}(M)$.
However the assumption is not so strong in practice.
For example,
\begin{exa}\label{exa1}
For any $G$-module $M$, we think of $M$
as an $\As(Z)$-module via the map $\mathcal{E}$ in \eqref{bba}. 
Then the identity $e_{z_1} \cdots e_{z_m} = \mathrm{id}_M$ follows from the condition $z_1 \cdots z_m = 1_G$.
\end{exa}

\begin{exa}\label{exa1222}
Let $G=\M_g$ and $Z=\mathcal{D}_g $ be as in Example \ref{LFrei}.
If $g \geq 3$, consider the two elements of the forms
$$ \kappa_{3\textrm{-}\mathrm{chain }} := (e_{c_1}e_{c_2}e_{c_3})^4 e_{d}^{-1}e_{d'}^{-1}, \ \ \ \ \ \kappa_{\rm lantern} := e_{c_1}^{-1}e_{c_3}^{-1}c_{c_5}^{-1}e_{b_{3}}^{-1} e_{b_1}e_{b_2}e_{d'} \in \As(\D), $$
where $b_i, \ c_i,$ and $ d^{(')}$ are the Dehn twists along the curves $\beta_i, \ \gamma_i,$ and $ \delta^{(')} $ in Figure \ref{tg261}, respectively;
Otherwise, if $g=2$, define $\kappa_{3\textrm{-}\mathrm{chain }} $ to be $(e_{c_1}e_{c_2}e_{c_3})^4 e_{c_5}^{-2} $,
and $\kappa_{\rm lantern} $ to be $\mathrm{1}_{\As (\D)}$.
As is well-known, $ \mathcal{E} (\kappa_{3\textrm{-}\mathrm{chain }} )$ and $ \mathcal{E} (\kappa_{\rm lantern} ) $ are the identity in $\mathcal{M} _g$,
which are commonly called the {\it 3-chain relation} and the {\it lantern relation}, respectively.
As is shown, 
the assumption $e_{z_1} \cdots e_{z_m}= \mathrm{id}_M$ is detected by
two elements $ \kappa_{3\textrm{-}\mathrm{chain }} $, $\kappa_{\rm lantern} $ and the signature of the Lefschetz fibrations;
see the last section in \cite{Nos1} for the details.
\end{exa}

\begin{figure}[h]
$$
\begin{picture}(220,65)
\put(-112,34){\pc{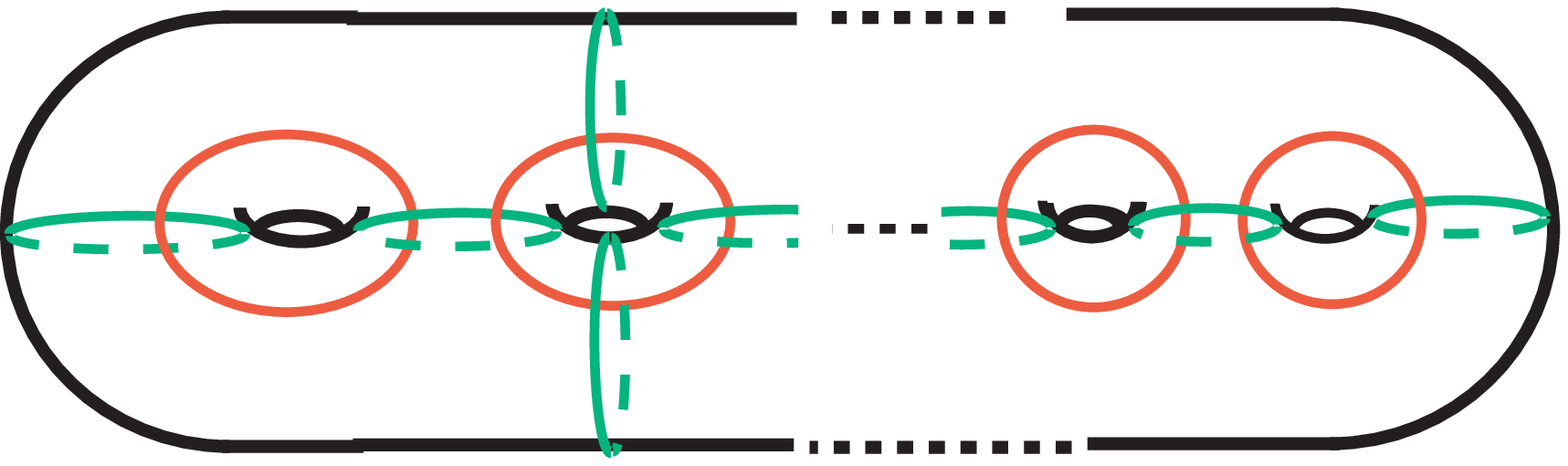}{0.36644466755}}

\put(132,34){\pc{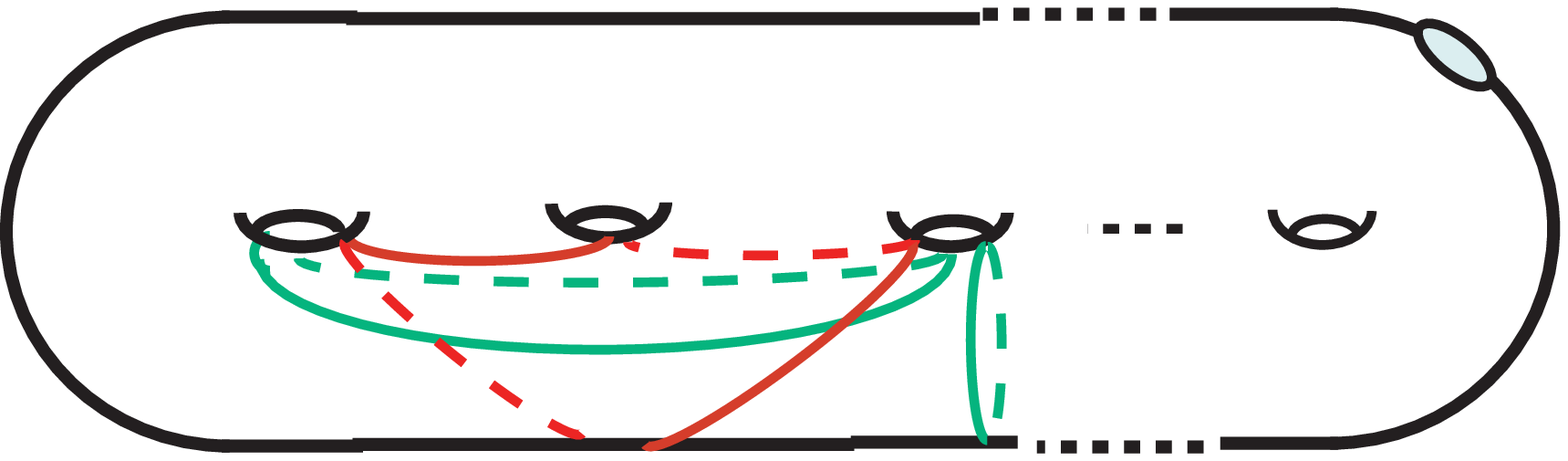}{0.36644466755}}

\put(-118,38){\large $\gamma_1$}
\put(-85,26){\large $\gamma_2$}
\put(-58,32){\large $\gamma_3$}
\put(-27,28){\large $\gamma_4$}
\put(10,25){\large $\gamma_{2g-2}$}

\put(-19,48){\large $\gamma_5$}
\put(-48,20){\large $\delta'$}
\put(-46,53){\large $\delta $}
\put(42,25){\large $\gamma_{2g}$}
\put(78,40){\large $\gamma_{2g+1} $}

\put(201,3){\large $\beta_1 $}
\put(190,44){\large$\beta_2$}
\put(256,26){\large $\beta_3$}

\end{picture}
$$

\vskip -1.225pc
\caption{\label{tg261} Generators of the group $\M_g$ with $g \geq 2$ and, the curves in the lantern relation with $g \geq 3$.}
\end{figure}
Under the assumption, 
we explain a topological interepretation of $\mathcal{Q}_{\psi,\ell} $ (see Proposition \ref{aa1133c}). 
which states a cohomological interpretation of
the bilinear form $\mathcal{Q}_{\psi,\ell} $ in the context of cup products.
Precisely, assume the identity $ e_{z_1} \cdots e_{z_m}= \mathrm{id}_M$  
and fix a tubular neighborhood, $D_i \ ( \cong \R^2)$, of the $i$-th point $\{ b_i\}$.
Let $Y$ be the 2-sphere $S^2$ with the $m$-open discs $D_1 \sqcup \cdots \sqcup D_m $ removed. 
Considering the action of $\pi_1(Y) $ on $M$ induced from that of $\pi_1(S^3 \setminus T_{m,m})$ and recalling \eqref{aa13},
we can define the cohomology $ H^*( Y,\partial Y ; M )$ with local coefficients $M.$
Furthermore, we write $M^{\rm op} $ for the $A$-module obtained by composing the module structure of $M$ with the involution.
\begin{prop}[{\cite[Proposition 4.7]{Nos3}}]\label{aa1133c}
Let $ e_{z_1} \cdots e_{z_m}= {\rm id} \in \End(M)$.
Then, the diagonal map $ M \ra \Ker ( \Gamma_{\mathbf{z}}) $ is a splitting injection, and
the cokernel is isomorphic to the relative cohomology $H^1(Y, \partial Y; M )$.

Moreover, the bilinear form $\mathcal{Q}_{\psi,\ell} $ coincides with the following composite, up to sign:
$$ H^1( Y,\partial Y ; M )^{\otimes 2} \stackrel{\smile }{\lra} H^2( Y ,\partial Y ; M^{\rm op }\otimes M )
\xrightarrow{ \ \langle \bullet , [(Y ,\partial Y )] \rangle \ } ( M^{\rm op } \otimes M )_{\As(Z)} \xrightarrow{ \ \psi(\bullet, \bullet) \ }A . $$
Here, the first map is the cup product, and the second is the pairing with
the fundamental class $[(Y ,\partial Y )] \in H_2(Y ,\partial Y ;\Z )\cong \Z $ in trivial coefficients,
and the third is the evaluation with $ \psi$.
\end{prop}
\noindent
As a result, the author \cite{Nos3} showed some properties of 
the bilinear form $ \mathcal{Q}_{\psi, \ell }$: 
\begin{prop}[{\cite[\S 4.3]{Nos3}}]\label{procol31}
Fix an $m$-tuple $ (z_1, \dots, z_m ) \in Z^m$ with the assumption $e_{z_1} \cdots e_{z_m}= \mathrm{id}_M \in \mathrm{End}(M)$.
\begin{enumerate}[(I)]
\item For any $\ell$ with $1 \leq \ell \leq m$, the equality $\mathcal{Q}_{\psi, 1 } = \mathcal{Q}_{\psi, \ell }$ holds as a map.
\item
If $\psi$ is hermitian, then the bilinear form $\mathcal{Q}_{\psi, \ell }$ is skew-hermitian.
Furthermore, if $\psi$ is skew-hermitian, then $\mathcal{Q}_{\psi, \ell }$ is hermitian.
\end{enumerate}
\end{prop}
\noindent
Since $\mathcal{Q}_{\psi ,\ell } $ does not
depend on $\ell$ under the assumption, we later only discuss $\mathcal{Q}_{\psi ,1 } $ with $\ell=1$, and denote it by $\mathcal{Q}_{\psi} $.
We further review a reduction of $\mathcal{Q}_{\psi} $, which implies
that a large part of $ \mathrm{Ker}( \Gamma_{\mathbf{z} }) $ is contained in the kernel of $\mathcal{Q}_{\psi} $; Precisely,
\begin{prop}[{\cite[Proposition 4.7]{Nos3}}]\label{pprop1}
Consider an element $ \mathbf{x} \in M^m$ of the form
$$ (x, \dots, x) , \ \ \ {\rm or }\ \ \ \ (x_1, \dots,x_m ) $$
for some $x \in M$ and $ x_i \in \Ker (1 -e_{z_i }) \subset M$.
Then, $ \mathbf{x}$ lies in the kernel $\mathrm{Ker}(\oplus_k \Gamma_{\mathbf{z},k }) $=$\mathrm{Ker}( \Gamma_{\mathbf{z} }) $,
and the vanishing $ \mathcal{Q}_{\psi,\ell} ( \mathbf{x}, \mathbf{y})= \mathcal{Q}_{\psi,\ell } ( \mathbf{y}, \mathbf{x})= 0$ holds for any $ \mathbf{y} \in \mathrm{Ker}( \Gamma_{\mathbf{z} }) $.
\end{prop}

\subsection{Unimodularity of the invariants of bilinear forms}\label{ss21192}
Next, we will discuss the unimodularity of $\mathcal{Q}_{\psi} $.
According to Proposition \ref{pprop1}, we shall regard $ \mathcal{Q}_{\psi}$ as a bilinear form on 
the quotient module of the form
\begin{equation}\label{mzmz} \underline{M}_{\mathbf{z} }:= \mathrm{Ker}( \Gamma_{\mathbf{z} }) / \bigl( \mathrm{Diag}(M) + ( \bigoplus_{ i: \ 1 \leq i \leq m } \Ker(1- e_{z_i}:M \ra M ) ) \bigr).
\end{equation}
To analyse the unimodularity in more details, 
we assume that $A$ is a Dedekind domain (because every known quantum representation is closed under a Dedekind domain; see \cite[Theorem 1,1]{Gil}), 
and $M$ is a finitely generated projective $A$-module.
Let $\Tor(M)$ denote the torsion submodule of $M$.
To the end, we suppose the following basic results on Dedekind domains:
\begin{enumerate}[(i)]
\item Every finitely generated $A$-module $M$ admits an ideal $\mathcal{I} \subset A$ and a unique integer $k \in \mathbb{N}$ for which
the $A$-isomorphism $M \cong A^k \oplus \mathcal{I} \oplus \Tor(M) $ holds. 
\item (Semi-heredity) Every $A$-submodule of a projective $A$-module is also projective.
\end{enumerate}
In particular, the canonical projection $ \mathrm{Ker}( \Gamma_{\mathbf{z} }) \ra \underline{M}_{\mathbf{z} }/\Tor (\underline{M}_{\mathbf{z} }) $
splits since the image is projective by (ii).
Therefore, denoting by $ M_{\mathbf{z} } \subset \mathrm{Ker}( \Gamma_{\mathbf{z} }) $ the summand obtained from the splitting, 
we have a decomposition 
$$ \mathrm{Ker}( \Gamma_{\mathbf{z} }) \cong M_{\mathbf{z} } \oplus (\mathrm{Ker}( \Gamma_{\mathbf{z} }) / M_{\mathbf{z} } )$$
as projective modules.
Then, by Proposition \ref{pprop1}, the vanishings $ \mathcal{Q}_{\psi} ( \mathbf{x}, \mathbf{y})= \mathcal{Q}_{\psi} ( \mathbf{y}, \mathbf{x})= 0$ hold for any
$ \mathbf{x} \in \mathrm{Ker}( \Gamma_{\mathbf{z} }) $ and $ \mathbf{y} \in \mathrm{Ker}( \Gamma_{\mathbf{z} })/ M_{\mathbf{z} } $.
Thus we shall address non-degeneracy restricted on the summand $ M_{\mathbf{z}} $: 
\begin{thm}\label{aa333c}
Let $A$ be a field or a Dedekind domain, and $M$ a finitely generated free $A$-module. 
Fix an $m$-tuple $ (z_1, \dots, z_m ) \in Z^m$ with $e_{z_1} \cdots e_{z_m}= \mathrm{id}_M $.
Take the projective submodule $ M_{\mathbf{z} } \subset \mathrm{Ker}( \Gamma_{\mathbf{z} }) $ explained above.
\begin{enumerate}[(I)]
\item If the $\As(Z)$-invariant bilinear function $\psi:M^2 \ra A$ is non-degenerate,
so is the resulting bilinear form $\mathcal{Q}_{\psi} $ restricted on the direct summand $M_{\mathbf{z} }$. 
\item Assume that $\psi$ is unimodular. Then, the restriction $\mathcal{Q}_{\psi} $ on the summand $M_{\mathbf{z} }$ is unimodular if and only if
the following map $\chi^*$ from a cokernel $\Coker(\lambda^{*})$ is injective:
$$ \chi^*: \Coker \Bigl( \Ext^1_A ( \frac{M}{ \mathcal{A}_{\mathbf{z}}}, A) \stackrel{ \lambda^{*}}{\lra} \Ext^1_A ( \bigoplus_{ 1 \leq i \leq m} \frac{M}{ \Ker(1-e_{z_i}) }, \ A ) \Bigr)\lra \Ext^1_A \bigl( \frac{M^{m-1}}{ \mathcal{B}_{\mathbf{z}} }, A \bigr) ,$$
where the maps $\iota^* $ and $ \kappa^* $ are induced by the homomorphisms
\[ \chi:M^{m-1}\lra M^m ; \ \ \ (a_1, \dots, a_{m-1}) \longmapsto (a_1, \dots, a_{m-1}, -a_1 -a_2 - \cdots -a_{m-1}) , \]
\begin{equation}\label{suzuki2} \lambda :M^{m}\lra M; \ \ \ \ \ \ \ \ (b_1,\dots,b_m) \longmapsto b_1 + \cdots +b_m, \end{equation}
respectively, and the submodules $ \mathcal{A}_{\mathbf{z}} $ and $\mathcal{B}_{\mathbf{z}} $ are
defined to be the followings:
\[ \mathcal{A}_{\mathbf{z}}:=\{\ b_1 \cdot ( 1- e_{z_1}) + \cdots + b_{m-1} \cdot (1 - e_{z_{m-1}}) \in M \ | \ b_1, \dots, b_{m-1} \in M \ \} ,\]
\[ \mathcal{B}_{\mathbf{z}}:= \{ \ ( b- b \cdot e_{z_1}, \dots, b - b \cdot e_{z_{m-1}}) \in M^{m-1}\ | \ b\in M \}. \]
\end{enumerate}

\end{thm}
\noindent
The proof of Theorem \ref{aa333c} will appear in \S \ref{ss2092}.
Since the statement (II) seems a little complicated, we should mention an applicable corollary, which is immediately obtained from Theorem \ref{aa333c}:
\begin{cor}\label{aa3333c}
With notation as above, if $\psi$ is unimodular and all the cokernel $ \Coker(1- e_{z_i} )$ with $i \leq m$ has no torsion,
then the restriction $\mathcal{Q}_{\psi} $ on the summand $M_{\mathbf{z} }$ is unimodular.
\end{cor}






\section{Invariants from symplectic representations}\label{ss2132}

We will study the bilinear forms $\mathcal{Q}_{\psi}$ obtained from
the standard symplectic representations, and give some applications.
Section \ref{ssg2} examines the rational part. 
To this end, Section \S \ref{ss2112} explains an outline for computing $\mathcal{Q}_{\psi}$.

\subsection{Preliminaries to compute the bilinear forms}\label{ss2112}
We first describe an outline for computing the isomorphisms class of $\mathcal{Q}_{\psi}$, on the basis of some knowledge in \cite{HM}.
This section assumes
that $M$ is a finitely generated free $A$-module.

Following Theorem \ref{aa333c},
we will discuss $\mathcal{Q}_{\psi}$ over a Dedekind domain $A$.
Recall the projective submodule $M_{\mathbf{z} } \subset M^m $ defined in \S \ref{ss21192}.
Then, as the properties (i)(ii) of Dedekind domains, this $ M_{\mathbf{z} }$ is a direct summand of $M^m$, and is isomorphic to
$A^{N-1} \oplus \mathcal{I} $ for some $N \in \mathbb{N}$ and ideal $\mathcal{I} \subset A $.
Therefore, we can choose a basis $v_1, \dots, v_{N} $, and 
we therefore formulate the bilinear form $ \mathcal{Q}_{\psi}$
as the $N \times N$-matrix, $W$, with entries $ \mathcal{Q}_{\psi}(v_i,v_j)\in A$.
Here we remark, thanks to Theorem \ref{aa333c} (i), that this $W$ is equivalent to the original $ \mathcal{Q}_{\psi} $,
and that, if $\psi$ is non-degenerate, then this $W$ is non-degenerate, i.e., $\det(W) \in A \setminus\{ 0\}$. 
In practice, although 
the ranks of $W$ are often large,
we can list such $W$ by the help of computer programs found
in software such as Mathematica.
Then, using the study of non-degenerate (or unimodular) bilinear forms over some rings (see, e.g., \cite{HM}),
we can sometimes determine some isomorphism classes of $W$ and obtain something quantitative from
the class (e.g., rank, signature) in some cases. 

As an example, we particularly comment the forms $ \mathcal{Q}_{\psi} $ when $A = \Z$ and $\psi$ is anti-symmetric and unimodular.
Theorem \ref{aa333c} implies that the matrix $W$ is symmetric and non-degenerate.
Moreover, if the condition in Corollary \ref{aa3333c} holds and
the matrix $W$ is indefinite, then
we may compute only the eigenvalues of $W$ and check whether the all
diagonal entries $ \mathcal{Q}_{\psi} (v_i,v_i)$ are even or not.
Actually, by virtue of Serre's classification theorem (see \cite[Chapters I, II]{HM} for details),
one can determine the bilinear form $W$ up to bilinear isomorphisms over $\Z$.
Furthermore, we also mention some definite matrices;
Actually, we see later some definite $W$'s.
However, the ranks are fortunately $\leq 16$; by the classifications of the definite unimodular bilinear forms of small rank (see \cite[Chapter II.6]{HM}),
we later determine some isomorphism classes of $\mathcal{Q}_{\psi}$ following the above line;
see \S \ref{ssg2} for concrete calculations.

For \S \ref{s3s331wlw}, we also comment coefficient rings of the forms $A= \mathbb{F}_p [T]/(T^p) $ for some prime $p \in \mathbb{N} $.
In general, it is hard to determine isomorphism classes of $\mathcal{Q}_{\psi}$, even if $\mathcal{Q}_{\psi}$ is symmetric. 
However, as a special case, if the image of $\mathcal{Q}_{\psi}$ is over the ideal $(T^{p-1})$,
the class of $ \mathcal{Q}_{\psi}$ is detected by the study of the Witt group $ W(\mathbb{F}_p )$ (see \cite[\S I.7 and Appendixes]{HM});
For example, if $p=2$, the class of $ \mathcal{Q}_{\psi}$ is detected by the
rank and the parity of diagonal entries $ \mathcal{Q}_{\psi} (v_i,v_i)$ that is
called Arf invariant.



\subsection{Computations in symplectic case}\label{ssg2}
Henceforth, we will compute the invariants of some Lefschetz fibrations
by means explained in the preceding section.
So, throughout this section, we assume that $ G$ is the mapping class group $\mathcal{M}_g$ with $g \geq 2$,
and $Z$ is the set $\D$ of the Dehn twists as in Example \ref{LFrei}.
This subsection discusses the standard symplectic representation of $\M_g$.
Namely, $M= H_1(\Sigma_g ;A )=A^{2g}$
and we set $\psi =\omega_{A} : M ^2 \ra A$ as the symplectic form.

We start by establishing some terminologies. 
For an $m$-tuple $(z_1, \dots, z_m) \in (\D)^m$,
let $E_{\mathbf{z}} $ be the (4-dimensional) total space of the corresponding fibration. 
The symbol $\sigma $ denotes the signature of the closed 4-manifold $E_{\mathbf{z}}$, 
and the Euler characteristic of the space $E_{\mathbf{z}}$ is known to be $ 4 - 4g + m $ (see \cite{EN,Ozb,Matsumoto}).
The fibration $ E_{\mathbf{z}} \ra S^2 $ 
is said to be {\it of type} $(m_{\rm ns},m-m_{\rm ns} ) $ if $m_{\rm ns}$ is
the number of $z_i$'s contained in $\DD$ \eqref{defdg11}.
Furthermore, for another fibration $\eta$
associated with $(z_1', \dots, z'_{m'}) \in (\D)^{m'}$,
{\it the (twist) fiber sum along $h \in \M_g$} is
defined to be the fibration arising from the $(m+m')$-tuples
$(z_1, \dots, z_m, h^{-1}z_1'h , \dots, h^{-1}z'_{m'}h) \in (\D )^{m+m'}$,
for which we write $ \xi \sharp_{h } \eta $.

In addition, we will explain some examples of Lefschetz fibrations.
Consider three tuples of the forms
\[ \xi_1 : \ (c_1 \cdot c_2 \cdots c_{2g} \cdot c_{2g+1}^2 \cdot c_{2g}\cdot c_{2g-1} \cdots c_1 )^2, \]
\[ \xi_2 : \ (c_1 \cdot c_2 \cdots c_{2g} \cdot c_{2g+1} )^{2g+2},\]
\[ \xi_3 : \ (c_1 \cdot c_2 \cdots c_{2g} )^{4g+2} .\]
Here $c_i$ is the Dehn twist of the curves $\gamma_i $ in Figure \ref{tg261}. 
Then, they are well-known to be identities in $\M_g$ as the chain relation
(The associated Lefschetz fibrations are understood well; see \cite[\S 4.1]{EN}).
Furthermore, as seen in Table \ref{G3},
we also deal with other Lefschetz fibrations.


Next, we comment rationally the rank and unimodularity of the bilinear forms $ \mathcal{Q}_{\omega_{\Z}}$,
and show that $ \mathcal{Q}_{\omega_{\Z}}$ includes rationally topological information of Lefschetz fibrations as follows: 
\begin{thm}\label{akdsj} Let $A=\Z$. 
Let $\xi_{\mathbf{z}} $ be a Lefschetz fibration $ E_{\mathbf{z}} \ra S^2$ associated with an $m$-tuple $\mathbf{z} $ of type $(m_{\rm ns},m- m_{\rm ns})$.
Take the homomorphism $\Gamma_{\mathbf{z}}$ in Definition \ref{aa11c} and the free module $\underline{M}_{\mathbf{z}}$ explained in Theorem \ref{aa333c}.
Then, the ranks are computed as
$$ \mathrm{rk} (\underline{M}_{\mathbf{z}} ) = m_{\rm ns} -4g + 2 b_1( E_{\mathbf{z}}), \ \ \ \ \ \ \ \mathrm{rk} (\Ker \Gamma_{\mathbf{z}} ) = 2gm-2g+  b_1( E_{\mathbf{z}}), $$
where $ b_1( E_{\mathbf{z}})$ is the first betti number of the 4-manifold $ E_{\mathbf{z}}$.

Moreover, the signature of the resulting symmetric form $ \mathcal{Q}_{\omega_{\Z}} $ is equal to $\sigma_{\mathbf{z}} -m+m_{\rm ns} $.
\end{thm}
\noindent
We defer the proof until \S \ref{ss2092}.
As a corollary, by the unimodularity in Corollary \ref{aa3333c}, we immediately have the following:
\begin{cor}\label{akdsj1}
The bilinear form $ \mathcal{Q}_{\omega_{\Z}}$
is equivalent to a unimodular symmetric bilinear form on $\Z^{m_{\rm ns} -4g + 2 b_1 }$ of signature $\sigma_{\mathbf{z}} -m+m_{\rm ns} $.
\end{cor}
\noindent

Accordingly, following the outline mentioned in \S \ref{ss2112}, if
knowing whether $ \mathcal{Q}_{\omega_{\Z}} $ is of type I or II,
we succeeded in determining the isomorphism class of $\mathcal{Q}_{\omega_{\Z}} $.
In fact, here is Table \ref{G3} on
the resulting forms of some Lefschetz fibrations with $g=2$ and $g=3$; 
Especially, the column on the symbol $\mathcal{Q}_{\omega_{\Z}} $
is with respect to isomorphism classes of the bilinear forms associated with the symplectic form $\omega_{\Z}$. 
Here $E_8$ is the famous Cartan matrix as a unimodular quadratic form of rank $8$
and of signature $-8$, and
$\mathcal{H}_c$ is the hyperbolic form $ {\small \left(
\begin{array}{cc}
0 & c \\
c & 0
\end{array}
\right)} $ for some $c \in A $.
As seen in the column, the type of $ \mathcal{Q}_{\omega_{\Z}} $ is frequently useful:
Actually, the pairs ($ \xi_{PJ}$, $\xi_{RI}$) and
$(\xi_1 \sharp_{c_1} \xi_1, \xi_3 ) $ with $g=2$ are shown to be non-isomorphic by the types (though these results are
proved also by considering the images of the tuple $ (z_1, \dots, z_m)$ in the modular group $ Sp(2g;\Z) $).
Here we shall cite a stronger result of H. Sato \cite{Sato}: Precisely, the total spaces associated with $\xi_{PJ}$ and $\xi_{RI}$
are non-diffeomorphic as a result in studying the (relative) Kodaira dimensions.


\begin{table}[htbp]
\large
\begin{center}
\begin{tabular}{|c|c| c|c|c|c|c|} \hline
Fibration & $g$ & Type& $\sigma$ & $ \mathcal{Q}_{\omega_{\Z}}$ &$ \mathcal{Q}_{\omega_{\rm spin}^{\rm odd}}$ & $ \mathcal{Q}_{\omega_{\rm spin}^{\rm ev}}$ \\ \hline \hline
$\xi_1$& 2 & (20,0) &$-12$ & $ (-1)^{12} 0^{64}$ & $ \mathcal{H}_{ {\tt y}}^{28} 0^{ 87}$ & $ \mathcal{H}_{ {\tt y}}^{50} 0^{ 141}$ \\ \hline
$ \xi_2$& 2 & (30,0)& $-18$ & $(-1)^{20} 1^2 0^{94}$ & $ \mathcal{H}_{ {\tt y}}^{48} 0^{ 127}$ & $ \mathcal{H}_{ {\tt y}}^{80} 0^{ 211}$ \\ \hline
$ \xi_3 $&2 & (40,0)& $-2$4& $ \mathcal{H}_1^4 (E_8)^3 0^{124}$ & $ \mathcal{H}_{ {\tt y}}^{70} 0^{ 166}$ & $ \mathcal{H}_{ {\tt y}}^{108} 0^{ 283}$ \\ \hline
$\xi_1 \sharp_{\rm id} \xi_1$ &2 & (40,0) & $-24$& $(-1)^{28} 1^4 0^{124}$ & $ \mathcal{H}_{ {\tt y}}^{68} 0^{ 169}$ & $ \mathcal{H}_{ {\tt y}}^{110} 0^{ 281}$ \\ \hline
$\xi_{RI}$ & 2 & (28,1)& $-17$ &$ 1^{2} (-1)^{18} 0^{92}$ & $ \mathcal{H}_{ {\tt y}}^{56} 0^{ 112}$& $ \mathcal{H}_{\tt y}^{74} 0^{207}$ \\ \hline
$\xi_{PJ}$ & 2 & $(28,1)$ & $ -17 $ & $\mathcal{H}_1^2 (E_8)^{2} 0^{92}$ & $ \mathcal{H}_{ {\tt y}}^{54} 0^{ 115}$& $ \mathcal{H}_{\tt y}^{72} 0^{209}$\\ \hline \hline
$\xi_1$ & 3 &$(28,0)$ & $ -16 $& $(E_8)^{2} 0^{146 }$ & $ \mathcal{H}_{\tt y}^{114} 0^{643}$ & $ \mathcal{H}_{\tt y}^{238} 0^{735}$\\ \hline
$\xi_1\sharp_{\rm id} \xi_1 $ &3 & $(56 ,0)$ & $ -32 $& $ \mathcal{H}_1^{6}(E_8)^{4}0^{286}$ & $ \mathcal{H}_{ {\tt y}}^{282} 0^{ 1259}$ & $ \mathcal{H}_{ {\tt y}}^{546}0^{1436} $ \\ \hline
$\xi_1\sharp_{ d } \xi_1 $ &3 & $(56 ,0)$ & $ -32 $& $ 1^{6} (-1)^{38 } 0^{286}$ & $ \mathcal{H}_{ {\tt y}}^{282} 0^{ 1259}$ & $ \mathcal{H}_{ {\tt y}}^{544}0^{1437} $ \\ \hline
$\xi_2$ &3 & $(56,0)$ & $ -32$& $ \mathcal{H}_1^{6}(E_8)^{4}0^{286}$ & $ \mathcal{H}_{ {\tt y}}^{282} 0^{ 1259}$ & $ \mathcal{H}_{ {\tt y}}^{546}0^{1436}$ \\ \hline
$\xi_3$ &3 & $(84,0)$ & $ -48 $& $ \mathcal{H}_1^{12}(E_8)^{6}0^{426}$ & $ \mathcal{H}_{ {\tt y}}^{452} 0^{ 1874} $ & $ \mathcal{H}_{ {\tt y}}^{854} 0^{2136 }$ \\ \hline
$\xi_2\sharp_{\rm id} \xi_1 $ &3 & $(84,0)$ & $ -48 $& $ \mathcal{H}_1^{12}(E_8)^{6}0^{426}$&$ \mathcal{H}_{ {\tt y}}^{450} 0^{ 1875} $ & $ \mathcal{H}_{ {\tt y}}^{854} 0^{2136} $ \\ \hline
$\xi_2\sharp_{ d} \xi_1 $ &3 & $(84,0)$ & $ -48 $& $ 1^{12}(-1)^{60} 0^{426}$&$ \mathcal{H}_{ {\tt y}}^{450} 0^{ 1876} $ & $ \mathcal{H}_{ {\tt y}}^{852} 0^{2137 }$ \\ \hline
$\xi_{CK}$ &3 & $(16,0)$ & $-8 $&$ E_80^{84}$ & $ \mathcal{H}_{ {\tt y}}^{54} 0^{373}$ & $ \mathcal{H}_{ {\tt y}}^{110} 0^{434}$ \\ \hline
$M_Q$ &3& $(44,1)$ & $ -25 $ & $ \mathcal{H}_1^{4}(E_8)^{3}0^{232}$ & $ \mathcal{H}_{ {\tt y}}^{210} 0^{1023}$ & $ \mathcal{H}_{ {\tt y}}^{414}0^{1172} $\\ \hline
$M_R$ &3& $(44,1)$ & $ -25$ & $ \mathcal{H}_1^{4}(E_8)^{3}0^{232}$ & $ \mathcal{H}_{ {\tt y}}^{210} 0^{1023}$ & $ \mathcal{H}_{ {\tt y}}^{414 }0^{1172} $ \\ \hline
\end{tabular}
\end{center}
\caption{The bilinear-form invariants associated with some $g$-genus Lefschetz fibrations with $g=2, \ 3 $.
Here global monodromies of the fibrations 
$\xi_{PJ}, \ \xi_{RI}$ with $g=2$ and $\xi_{CK}, \ M_Q, \ M_R $ with $g=3 $ are formulated in \cite[\S 4.1-2]{Endo};
Furthermore, the sixth and seventh columns $ \mathcal{Q}_{\omega_{\rm spin}^{\rm ev}}$ and $ \mathcal{Q}_{\omega_{\rm spin}^{\rm odd}}$
will be explained in \S \ref{a2k2lf}.
We write $\sharp_{ d}$ for the Dehn twist along the curve in Figure \ref{tg261}.
}
\label{G3}
\end{table}


To end the integral symplectic case,
we demonstrate from Table \ref{G3} that the bilinear form $\mathcal{Q}_{\omega_{\Z}}$ does not coincide with the usual cohomology ring structures with integral coefficients.
First, the form of the fibration $\xi_{CK}$ with $g=3$ is $E_8$
(cf. Donaldson's Theorem $A$ which implies that no differential closed 4-manifold admits the cohomology ring of the form $E_8$).
Furthermore,
while the simple connected space $\xi_3 $ with $g=2$ has no spin structure (use the Rokhlin theorem),
the form $\mathcal{Q}_{\omega_{\Z}}$ contain the matrix $E_8$. Thus, we pose a problem:
\begin{prob}\label{new1331}
Describe a sufficient and necessary condition for which
the unimodular form $ \mathcal{Q}_{\omega_{\Z}} $ is of type II.
\end{prob}
As seen in the diferrence between $\xi_2\sharp_{\rm id} \xi_1 $ and $\xi_2\sharp_{ d} \xi_1 $ in Table {G3}, 
we see that the torsion parts are not always topological.
Thus, we hope a usefulness of the torsion invariant. 
 




\section{Invariants from the quantum representations}\label{s3s331wlw}
In this section, we will explain
that the formulation of our invariants is applicable for
the quantum representations.

For the purpose, we briefly review the quantum representations.
First we fix a semi-simple Lie group $\mathcal{G}$ and an integer $k \in \Z $ called {\it level}.
Then E. Witten \cite{Wi} has predicted that the Chern-Simons quantum field theory of level $k$
yields a (2+1)-topological quantum field theory $\mathcal{Q}_{\mathcal{G},k }$ (with ambiguity of framings).
As a result, the restriction of $\mathcal{Q}_{\mathcal{G},k }$ on the
single object $\Sigma_g$ is regarded as a linear representation $\M_g$ of some $\mathbb{C}$-vector space $ V_{\mathcal{G},k}$,
which is commonly 
called the {\it quantum representation}.
However, to strictly formulate the TQFT in mathematical terms,
we have to equip the TQFT with ``$p_1$-structure" (see \cite[Appendix B]{BHMV} for the definition); thus,
by the reason of the obstruction of $p_1$,
the associated representations are often right modules not of $\M_{g,r}$, but of $\mathcal{T}_{g,r}$.
Here, $\mathcal{T}_{g,r} $ is a central $\Z$-extension of $\M_{g,r}$,
\begin{equation}\label{tyuusing2}0 \lra \Z \lra \mathcal{T}_{g,r} \xrightarrow{\ \mathrm{proj.}\ } \M_{g,r} \lra 0 \ \ \ \ \ \ \ {\rm (central \ extension) }. \end{equation}
Further, if $g \geq 3$, this $\mathcal{T}_{g,r}$ is universal in the sence of the group {\it co}homology $H^2(\M_{g,r} ;\Z) \cong \Z$; see \cite{FM}.
Moreover, since the $\mathcal{T}_{g,r}$-module $V_{\mathcal{G},k} $ is unitary by the topological invariance of the TQFT, we get
a non-singular $\mathcal{T}_{g,r} $-invariant Hermite bilinear form on $V_{\mathcal{G},k} $.
Furthermore, there is a conjecture that
the space $V_{\mathcal{G},k} $ would admit a lattice over the ring of integers in some number field $F \subset \mathbb{C}$
under which the action of $\mathcal{T}_{g,r} $ is closed.
Although we only refer the reader to \cite{BHMV,GMW,Ker,Wi} for more details,
we remark that 
the dimension of $V_{\mathcal{G},k} $ 
becomes exponentially more complicated as $k $-increase, that is, Velinder formula.


Thus, to specialize $\mathcal{Q}_{\psi}$ of small ranks,
we shall mention only easy examples of the TQFT.
First, if $ \mathcal{G}$ is the unitary group $U(1)$, then
the resulting module is the symplectic representation on $H_1(\Sigma_g)$,
and the bilinear form is the symplectic form, which we used in \S \ref{ss2132}.
In addition, if $ \mathcal{G}$ is $SU(2)$ or $SO(3)$,
the Kauffman bracket gives a description of the TQFT explicitly (see \cite{BHMV});
further, for any prime $p \in \mathbb{N}$,
the $SO(3)$-quantum representation of level $(p-1)/2$ is
over the ring of integers in the cyclotomic field $\Q(\zeta_{p})$ (see, e.g., \cite{Gil,GMW} and references therein),
though
the computations in $\Q(\zeta_{p})$ are often hardly sophisticated as usual in quantum topology. 
However, if the TQFT is closed under the integral closed ring $\Z[\zeta_{p^t }] $ for some $t \in \mathbb{N}$
such as $SO(3)$-case, the reduced TQFT via a base change 
\begin{equation}\label{opp} \Z[\zeta_{p^t}] \lra \mathbb{F}_p [ {\tt y}]/( {\tt y}^{p}); \ \ \ \ \ \zeta_{p^t} \longmapsto 1+{\tt y}
\end{equation}
is more manageable, together with a relation to Casson invariant modulo $p$ (see \cite{Ker}).
Such being the case, the next subsection will mainly focus on such reduced TQFTs.


\subsection{With the $SU(2)$-gauge of level $2 $. }\label{a2k2lf}
We will return to study Lefschetz fibrations.
Let $G = \M_g$ and $Z = \D$ as in Example \ref{LFrei}.
Then, as was shown \cite{Nos1}, $\As(\D) $ includes the universal central extension \eqref{tyuusing2} as follows:
\begin{thm}[{\cite{Nos1}}]\label{Prop12}
(I) If $ g \geq 3$, there is a group isomorphism $\As(\D) \cong \mathcal{T}_{g} \times \Z^{[g/2]+2} $.

\noindent
(II) If $g=2$, there are a $\Z$-central extension $ \mathcal{T}_{2} \ra \M_2$ and
an isomorphism $\As(\D) \cong \mathcal{T}_{2} \times \Z^2 $.
\end{thm}
As a result, for $g \geq 2$, the category of $\mathcal{T}_g$-modules is essentially equivalent to that of $ \As(\D)$-modules.
Hence, we can apply the quantum representations to Lefschetz fibrations.

In this section, we will focus on the quantum field with $SU(2)$-gauge of level $2 $
\footnote{
In terminologies of \cite{BHMV}, the TQFT
is written as ``$V_8 (\Sigma, q)$ with $r=4$",
where by $q$ we mean the set of all spin structures on $\Sigma$ either of odd parity or of
even parity.
},
and classify some examples of Lefschetz fibrations by the associated bilinear forms.
As is known from Velinder formula \cite[Theorems 1.6 (iii) and 7.17]{BHMV}, the quantum field produces two representations $V_{\rm spin}^{\rm ev}$
and $V_{\rm spin}^{\rm odd}$
of dimension
$$ \mathrm{dim}( V_{\rm spin}^{\rm ev}) = (2^g+1)2^{g-1}, \ \ \ \ \ \ \ \ \ \mathrm{dim}( V_{\rm spin}^{\rm odd}) =(2^g-1)2^{g-1}.$$
These $V_{\rm spin}^{\rm ev}$ and $V_{\rm spin}^{\rm odd}$ are modules of $ \mathcal{T}_g $ and of $\mathcal{T}_{g,1}$, and
are referred to as a quantization of spin structures on $\Sigma_g$ of even parity and of odd parity, respectively.

We now roughly explain the reason that the space $V_{\rm spin}^{\dagger }$ is
reduced to a free module over the ring $\Z[\zeta] \subset \mathbb{C}$,
where $\zeta$ is the $16$-th root of unity: Namely $\zeta:={\rm exp}(\sqrt{-1}/8\pi)$.
To begin, when $g=1$, we can see that the representation is conjugate to the following matrix presentation
with coefficients in $\Z[\zeta]$:
$$ \rho (c_1) := {\small \left(
\begin{array}{ccc}
1 & 0&( \zeta +\zeta^8 )/(1+ \zeta) \\
0 &- \zeta^3 & 0\\
0 & 0& - 1
\end{array}
\right)
} , \ \ \ \ \ \ \rho (c_2) := {\small \left(
\begin{array}{ccc}
0 & \zeta^2 & 0 \\
-\zeta^6 &0 & 0\\
\zeta^4(1- \zeta^{7} )/(1- \zeta) & \zeta^3( 1 - \zeta^{7} )/(1- \zeta) & - \zeta^3
\end{array}
\right)}
.
$$
For larger $g>1$,
it follows from the line of the paper \cite{GMW}
that the matrix presentation of $\rho (c_i)$ with $g>1$ is formulated
as a direct sum of that of $\rho(c_1)$ with $g=1$ via ``some change of coordinate system" by the 6j-symbols.
Then, thanks to the integrality of the 6j-symbols,
the quantum representation is closed over the ring $ \Z[\zeta]$ as required.
However, the details are unnecessary to this paper. Actually,
since we work out only $g \leq 3$ in this section, it is enough only to describe explicitly two matrix presentations of $\rho (c_1), \ \rho (c_2)$ over
$ \Z[\zeta]$.

We mention some properties of the representation.
First, the ring $ \Z[\zeta]$ is a PID, since the class number of the fraction field $\Q(\zeta ) $ is well-known to be 1.
However, we remark that, for $g >2 $, the $\mathcal{T}_g$-invariant Hermitian inner product $W_{\rm spin}^{\rm ev}$ obtained from the TQFT
is not always unimodular: In fact, even if $g=3$, we can check ${\rm det} (W_{\rm spin}^{\rm ev})= 2^{84} \cdot 23 \cdot 7583 \cdot 14110027847$.

Hence, for simplicity, let us consider the reduction on the Artin ring $\mathbb{F}_2[{\tt y}]/ {\tt y}^2$ as in \eqref{opp}.
Then, we can see that the $\mathcal{T}_{g,*}$-modules $V_{\rm spin}^{\dagger }$ are reduced to modules of $ \mathcal{M}_g $. 
Furthermore, we can find two bilinear forms from reduced representations as follows:
\begin{lem}\label{vw}
Let $g=2$ or $=3$.
Consider the $\mathbb{F}_2[{\tt y}]/ {\tt y}^2$-space of $ \mathcal{M}_g $-invariant bilinear forms on
the reduced representation $V_{\rm spin}^{\dagger } $ in coefficients $\mathbb{F}_2[{\tt y}]/ {\tt y}^2$.
Then, the space is spanned by two symmetric
bilinear forms $ \omega_{\rm spin}^{\dagger } $ and $ \nu_{\rm spin}^{\dagger } $
such that $ \omega_{\rm spin}^{\dagger }$ is non-singular and $ \nu_{\rm spin}^{\dagger } $ is singular.
\end{lem}
\noindent
Though we omit describing the proof, it can be done by a (long) direct calculation, thank to the matrix presentation of $ V_{\rm spin}^{\dagger } $.

However, Theorem \ref{aa333c} indicates that
the bilinear form $ \mathcal{Q}$ from the singular form $ \nu_{\rm spin}^{\dagger }$ enjoys
more interesting properties than that of $\omega_{\rm spin}^{\dagger } $.
In fact, in computer experience, the image of $ \mathcal{Q}_{\nu_{\rm spin}^{\dagger }}$
lies in the ideal $ ({\tt y}) \subset \mathbb{F}_2[{\tt y}]/ {\tt y}^2$;
Hence, we can easily determine the class of $ \mathcal{Q}_{\nu_{\rm spin}^{\dagger }} $ as mentioned in \S \ref{ss2112}.
In doing so, here are the fifth and sixth columns in Table \ref{G3} concerning the classes of Lefschetz fibrations.
As a result, we obtain the following:
\begin{prop}\label{vbbw} Fix $s,t \in \Z$.
The triples of the fibrations $(\xi_2 \sharp_{c_0 }^{2s} \xi_1, \ \xi_2 \sharp_{c_0 }^{2t-1} \xi_1,\ \xi_3 )$ are non-isomorphic.
\end{prop}

\begin{proof} For any Dehn twist $\tau_{\alpha}$, the twice $\bigl( \rho_{SU(2),2} ( e_{\alpha})\bigr)^2 $ is identity.
Hence, the bilinear forms of $\xi_2 \sharp_{c_0 }^{2s} \xi_1$ and $ \xi_2 \sharp_{c_0 }^{2t-1} \xi_1 $
are independent of $s,t \in \Z$. 
According to Table \ref{G3}, the above triple with $s=t=0$
are non-isomorphic. Hence, we have the conclusion as stated.
\end{proof}
However, 
as seen in Table \ref{G3}, the invariants $\mathcal{Q}_{\omega_{\rm spin}^{\dagger}}$ of lower level seem not so strong, in contrast to the Jones knot-polynomial and
with a recent result \cite{EK}.
For example, in the paper \cite{EK},
the two pairs $(\xi_1 \sharp \xi_1,\xi_2)$ and $ (\xi_P,\xi_Q)$ are shown to be non-isomorphic Lefschetz fibrations,
while the associated bilinear forms can not detect them.

To conclude, it is a problem for the future to compute and analyse the bilinear forms $\mathcal{Q} $ arising from the quantum representations with respect to other gauge of other level.

\section{Proofs of theorems in Sections \ref{ss21192}-\ref{ssg2}}\label{ss2092}
We will make the proofs of the statements in Sections \ref{ss21192}-\ref{ssg2}
that remain to be proved.
Since the proof is based on Proposition \ref{aa1133c}, the reader should reread it.

Furthermore, we prepare the universal coefficient theorem over Dedekind domains:
\begin{lem}\label{a2a3}
Let $A$ be a Dedekind domain, and $M$ be a finitely generated $A$-module.
Then there is an exact sequence
$$ 0\lra \Ext^1_A( H_{n-1}(Y;M),A ) \lra H^{n}( Y ;M ) \stackrel{ \mathcal{I}}{ \lra} \Hom ( H_n (Y ;M ), \ A ) \lra 0 . $$
Here $\mathcal{I} $ is the canonical map on the (co)chain group of $Y$.
Furthermore, $\mathcal{I} $ induces an isomorphism $ H^{n}( Y ;M )/\Tor \cong \Hom ( H_n (Y ;M ), \ A ).$
\end{lem}
\begin{proof}
Note that $\Ext^n$ with $n \geq 2 $ are zero since $A$ is of Krull dimension 1.
Thus, the universal coefficient spectral sequence (see \cite[\S 12.1]{McC}) collapses at $E_2$-term;
hence we have the exact sequence.
For the latter part, by the reason of (i) explained in \S \ref{ss21192},
$ \Hom ( H_n (Y ;M ), \ A ) $ is projective and $\Ext^1_A( H_{n-1}(Y))$ is torsion.
Thereby, we immediately obtain the isomorphism.
\end{proof}

\begin{proof}[Proof of Theorem \ref{aa333c}]
(I) 
We start by considering the long exact sequence associated with the pair $ ( Y,\partial Y ) $:
\begin{equation}\label{ajsk} 0 \lra H^0( Y ; M ) \lra H^0( \partial Y ; M ) \stackrel{\delta^*}{\lra} H^1( Y,\partial Y ; M ) \stackrel{p^*}{\lra} H^1( Y; M ) \stackrel{\iota^*}{\lra} H^1( \partial Y ; M ) .
\end{equation}
As is seen in \cite[Proposition 4.7]{Nos3}, the cokernel of $\delta^*$ can be identified with the quotient $ \underline{M}_{\mathbf{z} } $ in \eqref{mzmz}.
In particular, the module $ M_{z}$ is, by construction, regarded as $ \Coker (\delta^* )/\Tor$.

Based on the identifications,
we will reformulate the adjoint map of bilinear form $\mathcal{Q}_{\psi}$, and will show its injectivity.
Define a homomorphism
$$ \Upsilon: H^1( Y, \partial Y ; M ) \lra \Hom ( H^1( Y ; M ) , \ (M \otimes M)_{\pi_1(Y)})$$
by setting $\Upsilon (u) := \langle u \smile \bullet , [Y, \partial Y]\rangle \in (M \otimes M)_{\pi_1(Y)} $.
Hence, composing $ \Upsilon $ with the pairing of the bilinear form $\psi$ yields a homomorphism
\begin{equation}\label{aa33c} \Upsilon_{\psi}: H^1( Y, \partial Y ; M ) \lra \Hom ( H^1( Y ; M ) , \ A) .
\end{equation} 
Here, we should notice that the map $ \langle \bullet , [(Y ,\partial Y )] \rangle $
is an isomorphism by Poincar\'{e} duality on the first (co)homology in local coefficients;
thus, this $ \Upsilon_{\psi}$ is injective and the cokernel is torsion,
thanks to the Poincar\'{e} duality and the non-degeneracy of $\psi$.
Here, recall an elementary equality on the cup product:
$$ \langle u \smile p^*(v) , [Y, \partial Y]\rangle =0 \in (M \otimes M)_{\pi_1(Y)} , \ \ \ \ \ \
{\rm for \ any \ } u \in \Im (\delta^*), {\rm \ \ } v \in H^1( Y ; M ) .$$
Hence, by considering the restriction on $\Im (p^*)= \Coker (\delta^* ) $, 
the map $ \Upsilon_{\psi}$ in \eqref{aa33c} induces
\begin{equation}\label{aac3321}\Upsilon_{\psi}': \Coker (\delta^* )/\Tor \lra \Hom ( \Im (p^*) , A).
\end{equation}
By constructions and Proposition \ref{aa1133c}, this $\Upsilon_{\psi}'$ is the adjoint map of $\mathcal{Q}_{\psi}$.
Letting $K$ be the fractional field of $A$, since
$\Upsilon_{\psi} \otimes \mathrm{id}_K$ is injective as mentioned in \eqref{aa33c},
so is $ \Upsilon_{\psi}' \otimes \mathrm{id}_K $.
Hence, we have the injectivity of $ \Upsilon_{\psi}'$, which completes the proof for the non-degeneracy of $\mathcal{Q}_{\psi}$.




(II) Assuming the unimodularity of $\psi$, we will observe that of $\mathcal{Q}_{\psi}$.
To this end, consider the following commutative diagram:
$${\normalsize
\xymatrix{
H^1( Y, \partial Y , M) \ar[r]\ar[d]_{\Upsilon_{\psi} } & \Im (p^*)/\Tor \ar[r] \ar[d]^{\Upsilon_{\psi}' } & \ \ 0\ar[d] \ \ & \\
\Hom( H^1( Y ; M),A ) \ar[r] & \Hom( \Im (p^*),A) \ar[r] & \Ext^1 ( \Im (\iota^*), A) \ar[r]^{\iota^*} & \Ext^1 ( H^1( Y ;M ), A) ,
}}
$$
where the bottom sequence is the $\Ext$-functor obtained from \eqref{ajsk}, and the
horizontal sequences are exact.
Notice that, by Poincar\'{e} duality, the map $\Upsilon_{\psi} $ mentioned in \eqref{aa33c} is an isomorphism.
Therefore the Snake Lemma implies that the cokernel of $\Upsilon_{\psi}'$ is identified with
the kernel of the bottom right map $ \iota^* $.

We will analyse the kernel $\Ker(\iota^* )$,
because the unimodularity of $\mathcal{Q}_{\psi}$ is equivalent to
the injectivity of the map $ \iota^*$.
Here, note that, by 1-dimensional duality of $Y$,
we can regard the image $\Im(\iota^*)$ as
the kernel of the pushfoward $\iota_{*}: H_0(\partial Y;M) \ra H_0( Y;M).$

Therefore, we now analyze this $\iota_{*} $ by computing $H_0(\partial Y;M) $ and $H_0( Y;M) $.
A standard resolution of the free group $\pi_1(Y)$ after tensoring with $M$
is written as
$$ 0 \lra M \otimes \Z [\pi_1(Y)]^{m-1} \xrightarrow{ \ \ \partial_1 \ \ } M \otimes \Z [\pi_1(Y)] \lra 0 . $$
Here, the map $\partial_1$ sends $ a \otimes (g_1 , \dots, g_{m-1} )$ to $\sum_{i=1}^{m-1}( a g_i) \otimes (1 -g_{i})$.
Thereby, the 0-th and first cohomologies are given by
$$ H^0( Y;M) \cong M /\mathcal{A}_{\mathbf{z}} , \ \ \ \ H^1( Y ;M )\cong M^{m-1} /\mathcal{B}_{\mathbf{z}} ,$$
where the submodules $\mathcal{A}_{\mathbf{z}}$ and $\mathcal{B}_{\mathbf{z}}$
were defined in Theorem \ref{aa333c}.
Here, notice by definitions that the 0-th homology of the boundary forms
$H_0(\partial Y;M) \cong \oplus_{i=1}^m \Coker(1 -e_{z_i})$.
Then, the map $\iota_*: H_0(\partial Y;M) \ra H_0( Y;M) $ is identified
with the map $\lambda$
from the definition \eqref{suzuki2}.
Hence. in the Ext$^1$ terms, we have an isomorphism between $ \Ext^1 ( \Im (\iota^*), A) $ and the cokernel of $\chi^*$ mentioned in Theorem \ref{aa333c}.
In a similar fashion, 
we can identify the kernel of the map $ \Ext^1 ( \Im (\iota^*), A) \ra \Ext^1 ( H^1( Y ;M ), A) $
with the kernel of $\chi^* $.
In particular, the cokernel of $\Upsilon_{\psi}'$ is zero if and only if
the kernel of $ \chi^*$ is zero.
Hence $\Upsilon' $ is unimodular if and only if
$ \chi^*$ is injective.
This completes the proof.
\end{proof}
Changing the subject, we will show Theorem \ref{akdsj}.
\begin{proof}[Proof of Theorem \ref{akdsj}]
We start by setting up notation. Let $M= H^1(\Sigma_g;\Q)$.
Let $E_{\mathbf{z}}$ be the total space of $\xi$ as a closed 4-manifold
with a $C^{\infty}$-map $ \pi : E_{\mathbf{z}} \ra S^2$,
and let $E_{\mathbf{z}}^{\rm reg}$ be the complimentary space $ p^{-1}( Y ) \subset E_{\mathbf{z}}$.
Notice that, from the definition of Lefschetz fibrations, the restriction of $\pi : (E^{\rm reg}_{\mathbf{z}}, \partial E_{\mathbf{z}}^{\rm reg}) \ra (Y, \partial Y)$ is a relative $\Sigma_g$-fibration.

Next, we will compute the rank of $\Ker (\Gamma_{\mathbf{z}}) $.
Recalling the isomorphism $\mathcal{H}^1(Y, \partial Y;M ) \cong \Ker (\Gamma_{\mathbf{z}}) \oplus M$ in Proposition \ref{aa1133c},
we will focus on the Leary-Serre spectral sequence of $\pi$, which is written as
$$ E_2^{p,q} \cong \mathcal{H}^p( Y, \partial Y ; \ H^q(\Sigma_g ;\Q ) ) \ \Longrightarrow \ E_{\infty } \cong H^*(E^{\rm reg}_{\mathbf{z}}, \partial E_{\mathbf{z}}^{\rm reg} ;\Q). $$
Here we remark that if $q=1$, the coefficients in the $E_2$-term arise from the symplectic representation on $M= H^1(\Sigma_g;\Q)$.
Then, one can easily notices $ E_2^{0,*} \cong 0$, because
$$ \mathcal{H}^0(Y, \partial Y; H^q(\Sigma_g ;\Q ) ) \cong \mathcal{H}_2(Y; H^q(\Sigma_g ;\Q ) )\cong 0$$ by duality.
Therefore, the spectral sequence collapses at the $E_2$-page, leading to isomorphisms
\begin{equation}\label{a33sk1} H^2( E^{\rm reg}_{\mathbf{z}}, \partial E_{\mathbf{z}}^{\rm reg} ;\Q) \cong \mathcal{H}^1(Y, \partial Y;M ) \oplus \Q ,\end{equation}
\[ H^3( E^{\rm reg}_{\mathbf{z}}, \partial E_{\mathbf{z}}^{\rm reg} ;\Q) \cong \Q^{m-1} \oplus (M _{\pi_1(Y)}), \]
\[ H^4( E^{\rm reg}_{\mathbf{z}}, \partial E_{\mathbf{z}}^{\rm reg} ;\Q) \cong \mathcal{H}^2(Y, \partial Y; H^2(\Sigma_g ;\Q ) )\cong \mathcal{H}_0(Y ; H^2(\Sigma_g ;\Q ) ) \cong \Q . \]
Thus, for the attempt to compute $ \mathcal{H}^1(Y, \partial Y;M )$, we shall study the rational cohomology $ H^*( E^{\rm reg}_{\mathbf{z}}, \partial E_{\mathbf{z}}^{\rm reg} ;\Q)$.
Noticing the excision axiom
\begin{equation}\label{ajddsk1}H^*(E^{\rm reg}_{\mathbf{z}}, \partial E_{\mathbf{z}}^{\rm reg} ;\Q ) \cong H^*(E_{\mathbf{z}}, \ \pi^{-1}(D_1 \sqcup \cdots \sqcup D_m ) ;\Q ),
\end{equation}
and, denoting the union $ \bigsqcup_{i=1}^{ m} \pi^{-1}( D_i ) $ by $ W $, we now consider the long exact sequence
\begin{equation}\label{ajddsk1}
0\lra H^1( E_{\mathbf{z}}^{\rm reg} )\lra H^{1}(W ) \lra H^{2} (E_{\mathbf{z}}^{\rm reg}, \ W )\ra \cdots \ra H^{3} (E_{\mathbf{z}}^{\rm reg}) \lra H^{3} ( W )\lra 0 ,
\end{equation}
where the coefficients are trivial and rational.
It is worth noticing the isomorphisms
$$H^{1}(W ) \cong \Q^{2gm - m_{\rm ns} } , \ \ \ \ \ H^{2}(W ) \cong \Q^{m - m_{\rm ns} } $$
which immediately follow from Lemma \ref{a2a5dd523} below. 
Therefore, with notation $H^i(E_{\mathbf{z}} ) \cong \Q^{b_i}$, the exact sequence \eqref{ajddsk1} reduces to
$$ 0\ra \Q^{b_1} \ra \Q^{2gm - m_{\rm ns} } \ra H^{2} (E^{\rm reg}_{\mathbf{z}}, \partial E_{\mathbf{z}}^{\rm reg}) \ra \Q^{b_2} \ra \Q^{m- m_{\rm ns}} \ra \Q^{m-1} \oplus M _{\pi_1(Y)} \ra \Q^{b_1} \ra 0. $$
Recalling the Euler characteristic $\chi (E_{\mathbf{z}}) =m-4g+ 4$, we readily have $b_2= m-4g+2+2b_1$.
Moreover, noting $ M _{\pi_1(Y)} \cong \Q^{b_1} $ by Lemma \ref{a2a5dd523} again, the exact sequence therefore leads to 
\begin{equation}\label{a33sk221} \mathrm{dim}(H^{2} (E^{\rm reg}_{\mathbf{z}}, \partial E_{\mathbf{z}}^{\rm reg} ) ) = 2gm +b_1 +1. \end{equation}
Hence, by combining the assumption $\mathcal{H}^1(Y, \partial Y;M ) \cong \Ker (\Gamma_{\mathbf{z}}) \oplus M$ with \eqref{a33sk1},
the rank of $ \Ker (\Gamma_{\mathbf{z}}) $ turns out to be $2g(m-1)+b_1$ as desired.

Next, we will compute the rank of the quotient free $\Z$-module $ \underline{M}_{\mathbf{z}}$ defined in \eqref{mzmz}.
According to the proof of Theorem \ref{aa333c},
this module is isomorphic to the cokernel $ \Coker (\delta_* )$.
Notice $ \mathrm{rk}(\mathcal{H}^{1} (Y, \partial Y; M )) = 2gm +b_1$ from \eqref{a33sk221}.
Furthermore, noticing $ H^{0} ( \partial Y;M ) \cong \oplus_i^{m} \Ker( 1 - e_{z_i}) $ by definition,
the rank is $2gm - m_{\rm ns} $, since
the matrix presentation of the symplectic representation implies that, if $z \in \DD$, the rank of $1 - e_{z}$ is $1$;
otherwise the rank of $1 - e_{z}$ is zero.
Finally, noting $ {\rm rk}\mathcal{H}^{0} ( Y;M )={\rm rk}H_1( \Sigma_g ;\Z )_{\pi_1(Y)} =b_1$ from Lemma \ref{a2a5dd523},
the sequence \eqref{ajsk} implies the desired equality $ \mathrm{rk} M_{\mathbf{z}} = \mathrm{rk} \underline{M}_{\mathbf{z}} =m_{\rm ns} - 2g +2 b_1$.


Finally, we will analyze the signature.
Proposition \ref{procol31} says that $\mathcal{Q}_{\psi}$ is symmetric,
and that the signature of $\mathcal{Q}_{\psi}$ is
equal to that of the cup product form on $\mathcal{H}^1(Y, \partial Y;M ) $.
Furthermore, regarding the symplectic form $\omega$ as the ring structure on $H^1(\Sigma_g)$,
one can verify that the (relative) Leary-Serre spectral sequence of $\pi $ preserves the cup product (see \cite[\S 5.3]{McC} or \cite{Meyer}).
Hence the (relative) signature on $\mathcal{H}^1(Y, \partial Y;M ) $ is equal to
that of $E^{\rm reg}_{\mathbf{z}}$.
Moreover $\sigma (E^{\rm reg}_{\mathbf{z}}) = \sigma(E_{\mathbf{z}})+ m-m_{\rm ns} $ is known as a result of Novikov-Wall additivity formula (see
\cite{Ozb,EN}).
In summary, the signature of $\mathcal{Q}_{\psi}$ is
equal to $ \sigma(E_{\mathbf{z}})+ m- m_{\rm ns} $ as claimed.
\end{proof}

\begin{lem}\label{a2a5dd523}
Let $D_i \subset S^2$ be the $i$-th open disk explained above.
Then
$$ H_1 (\pi^{-1}D_i ;\Z) \cong \left\{
\begin{array}{ll}
\Z^{2g-1} , &\quad {\rm if \ } \gamma_i \in \DD, \\
\Z^{2g} , &\quad {\rm otherwise. }
\end{array}
\right. \ \ \ \ \ \ H_2 (\pi^{-1}D_i ;\Z) \cong \left\{
\begin{array}{ll}
\Z , &\quad {\rm if \ } \gamma_i \in \DD, \\
\Z^{2} , &\quad {\rm otherwise. }
\end{array}
\right.
$$
Further, $ H_n (p^{-1}D_i ;\Z)$ vanishes for $n \geq 3$.

In addition, the coinvariant $ H_1(\Sigma_g;\Z)_{\pi_1 (Y)}$ is isomorphic to the integral first homology of $E_{\mathbf{z}}.$
\end{lem}
\begin{proof} The former claim follows from
observing the 2-handle attaching associated with the singular value $b_i$ (see \cite[\S 1]{Ozb} for details).
In addition,
it can be easily seen that a Meyer-Vietoris argument
deduces the latter part on the homology $H_1(E_{\mathbf{z}};\Z) $.
\end{proof}
\subsection*{Acknowledgment}
The author expresses his gratitude to Naoyuki Monden, 
Kenta Hayano, Hisaaki Endo, and
Yoshihisa Sato for useful discussions on Lefschetz fibrations.
He is also grateful to R. Inanc Baykur and Hakho Choi for useful comments.

\appendix
\section{The fundamental quandle of the 2-sphere with $m$-points }\label{s3s3wsd}
We will show that the Hopf fibration $\mu: S^3 \ra S^2 $
yields a quandle isomorphism
between some fundamental quandles (Proposition \ref{isnsdvdef}).
This discussion is analogous to the Dirac belt trick in physical understanding.

We begin by reviewing quandles and giving examples.
A {\it quandle} \cite{Joy} is a set, $X$, is a non-empty set with a binary operation $(x, y) \rightarrow x\lhd y$ such that,
for any $x,y,z \in X,$ the equalities $ x \lhd x = x$ and $(x\lhd y)\lhd z = (x\lhd z)\lhd(y \lhd z)$ hold and
there exists uniquely $w \in X$ satisfying $w \lhd y = x$.
\begin{exa}[Conjugacy quandle]\label{conjex}
Every union of some conjugacy classes of a group $G$ is a quandle with the conjugacy operation $x\tri y:=y^{-1}xy$ for any $x,y \in G$.
For example, the subsets $\D, \ \DD $ of $\M_g$ in Example \ref{LFrei} are regarded as 
conjugacy quandles. 
\end{exa}
\begin{exa}[{Fundamental quandle \cite{Joy}}]\label{Dd}
Let $M$ be a connected oriented $C^{\infty}$-manifold without boundaries, and $N\subset M$ a submanifold of codimension $2$.
For simplicity, assume the case where a tubular neighborhood of $N \subset M$ is diffeomorphic to $D^2 \times N $, where $D^2$ is the 2-disk.
Fix two points $p_0 \in M\setminus N$ and $d_0 \in \partial D^2$.
We define $Q(M,N)$ to be the set of homotopy classes of all pairs $(D^2 \times \{y\}, \alpha)$, where $y$ runs over $N$ and
$\alpha$ is a path in $M \setminus N$ starting from $p_0$ and ending at $d_0$ (see, e.g., Figure \ref{tg2}).
The set $Q(M,N)$ carries a quandle structure, called {\it fundamental quandle}, with an operation defined by
\[ [(D_1 \times \{ y_1\} , \alpha_1)]\tri [(D_2 \times \{ y_2 \}, \alpha_2)] = [(D_1 \times \{ y_1\}, \alpha_1 \cdot_{p_0} \alpha_2^{-1} \cdot_{d_0} \overline{\partial D_2} \cdot_{d_0} \alpha_2 )]. \]
Here the symbol $\overline{\partial D}$ is the path obtained from the boundary $ \partial D$ by cutting at $d_0$, and
the symbol $\cdot_{y}$ means the connection of two paths at a point $y$. See \cite[\S 12--15]{Joy} for details.\end{exa}
\noindent
As examples, let $\mm$ be the set of $m$-points $ \{b_1, \dots ,b_m \}$ in $S^2$ with the same orientations.
Then the Hopf fibration induces a quandle homomorphism $ Q(S^3 ,T_{m,m}) \ra Q(S^2 ,\mm )$. Moreover
\begin{prop}\label{isnsdvdef} The Hopf fibration $\mu$ induces a quandle isomorphism $ Q(S^3 ,T_{m,m}) \cong Q(S^2 ,\mm )$.
\end{prop}


To show this,
we consider the diagram of $T_{m,m}$ in the left hand side of Figure \ref{tg25}.
Then, 
following from Wirtinger presentation for quandles (see \cite[\S 15]{Joy}), the fundamental quandle $Q(S^3 , T_{m,m})$ is presented by
\[ {\rm generators}: \ \ \ \ \ \ a_i \ \ \ \ \ \ \ \ \ \ \ \ (i \in \Z), \]
\begin{equation}\label{rel} {\rm relations}:\ \ \ \ \ \ (\cdots (a_i \tri a_{i+1}) \tri \cdots )\tri a_{i+m-1}= a_i, \ \ \ \ \ \ \ a_i=a_{m+i} \ \ \ \ \ (i \in \Z).\end{equation}
Here, for $1 \leq i \leq m$, the generator $a_i$ is represented by the path in the left of Figure \ref{tg25}.
Then the image $c_i := \mu (a_i) \in Q(S^2,\mm) $
is illustrated as the path shown in the right of Figure \ref{tg25}.

\begin{figure}[h]
$$
\begin{picture}(220,100)
\put(-75,45){\pc{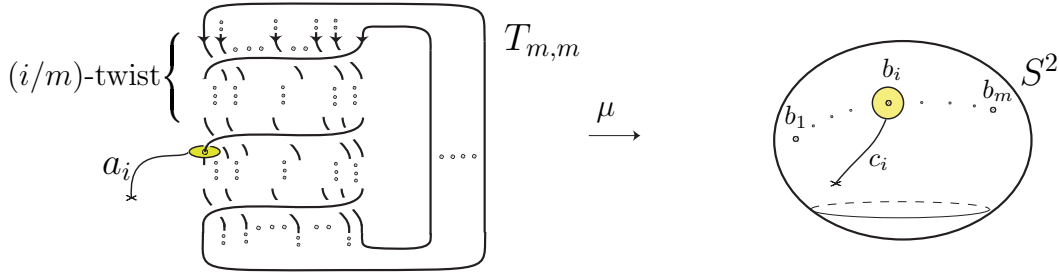}{0.2755}}

\put(-79,35){\Large $a_i$}

\put(73,82){\Large $T_{m,m}$}
\put(268,69){\Large $S^2 $}

\put(-115,69){\large $(i/m)$-twist}

\put(-59,55){\normalsize \rotatebox{90}{$ \overbrace{ \ \ \ \ \ \ \ \ \ \ }^{\ }$ }}

\put(108,57){\large $\mu$}

\put(254,64){\normalsize $b_m $}
\put(216,72){\normalsize $b_i$}
\put(179,54){\normalsize $b_1$}

\put(211,38){\normalsize $c_i$}

\end{picture}
$$

\vskip -0.5pc
\caption{\label{tg25} The quandle homomorphism induced from the Hopf fibration $\mu$.}
\end{figure}
\begin{proof}[Proof of Proposition \ref{isnsdvdef}] We will examine the image $Q(S^2 ,\mm)$ for more details.
Choose a disk $D^2 $ in $S^2$ containing the points $\mm $.
The inclusion $D^2 \subset S^2$ induces a quandle epimorphism $\kappa: Q(D^2 ,\mm) \ra Q(S^2 ,\mm)$ by definitions.
Consider the element $\overline{c_i} \in Q(D^2 ,\mm)$ drawn in the left of Figure \ref{tg2}.
By van-Kampen theorem for the fundamental quandles (see \cite[Theorem 13.1]{Joy}),
$ Q(D^2 ,\mm) $ is the free quandle generated by $\overline{c_i}$ (shown by induction on $m$),
and the kernel of $\kappa$ is generated by the equivalences in Figure \ref{tg2}.
The kernel corresponds to the relation in \eqref{rel};
thus the presentations of $Q(S^3 ,T_{m,m})$ and $Q(S^2 ,\mm) $ imply the desired bijectivity of $\mu$.
\end{proof}
\vskip -2.44pc
\begin{figure}[h]
$$
\begin{picture}(220,100)
\put(-46,35){\pc{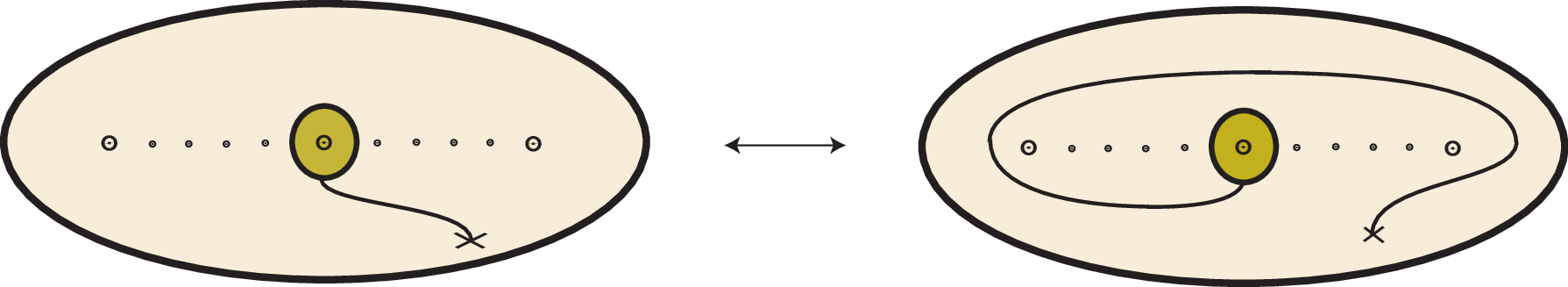}{0.3755}}

\put(66,50){\normalsize $b_m $}
\put(22,51){\normalsize $b_i$}
\put(-21,50){\normalsize $b_1$}

\put(78,62){\Large $D^2 $}
\put(268,62){\Large $D^2 $}

\put(31,18){\large $\overline{c_i}$}

\end{picture}
$$

\vskip -1.3944pc
\caption{\label{tg2} The relation in $ Q(D^2,\mm)$.}
\end{figure}

\vskip -3.00944pc

\

\section{Comparison with Meyer 2-cocycles and signature}\label{asss303k1}
This appendix compares Theorem \ref{akdsj} and some results of Meyer \cite{Meyer}, in particular, a group 2-cocycle of the $\M_g$.

We start a brief review of some results of Meyer \cite{Meyer} (See also \cite[Appendix A]{EN}).
Identifying the homology $M:= H^1(\Sigma_g;\R )$ with $\R^{2g}$,
let $\M_g$ act on this $ \R^{2g}$ as the symplectic representation.
For arbitrary two elements $A,B \in\M_g $, we consider the subspace
$$ V_{A,B} := \bigl\{ \ (x, y) \in \R^{2g} \times \R^{2g} \ \bigl| \ \ \ x \cdot (1- A ) + y \cdot (B - 1 ) = 0 \ \bigr\} $$
of the real vector space $\R^{2g} \times \R^{2g}$.
Recalling the symplectic form $ \omega :(\R^{2g})^2 \ra \R $,
let us define a bilinear form $\langle , \rangle_{A,B} : (V_{A,B} )^2 \ra \R$
by setting $ \langle (x_1, y_1 ), (x_2,y_2)\rangle_{A,B} := \omega \bigl( x_1 + y_ 1, \ y_2 \cdot (1 - B) \bigr) , $
We easily see this form $\langle , \rangle_{A,B}$ is symmetric; hence, we can define a map
$$ c_{ \M_g} : \M_g \times \M_g \lra \Z ; \ \ \ (A,B) \longmapsto \mathrm{Sing}( V_{A,B}, \ \langle , \rangle_{A,B} \ ). $$
Then Meyer showed that $c_{ \M_g}$ is a normalized group 2-cocycle of $\M_g$.

We now explain a comparison with the bilinear form $ \mathcal{Q}_{\omega}$ in Definition \ref{aa11c}.
To this aim, 
\begin{prop}\label{a12}
Let $G=Z= \M_g$, and let $M:= H^1(\Sigma_g;\R ) $.
For $A,B \in\M_g $, consider a 3-tuple $\mathbf{z} = (A,B, B^{-1}A^{-1}) \in Z^3$.
Then the splitting surjection $ M^3 \ra M^2$ which sends $(a,b,c)$ to $(a-b,b-c)$
gives rise to an isomorphism of bilinear forms
$$ \bigl( \Ker (\Gamma_{\mathbf{z}}) , \ \mathcal{Q}_{\omega} \bigr) \cong \bigl( V_{A,B} , \ \langle , \rangle_{A,B} \ \bigr) \oplus (M, 0). $$
\end{prop}
\begin{proof}Straightforward from the definitions.
\end{proof}
\noindent
Here we give a geometric interpretation of the isomorphism.
Let $Y$ be the 2-sphere with two 2-disks removed, and
and consider the action $\pi_1(Y) \curvearrowright H^1(\Sigma_g;\R )=M $ similar to the proof of Theorem \ref{akdsj} (see \S \ref{ss2092}).
Meyer \cite{Meyer} showed, using the simplicial cohomology,
that the space $V_{A,B} $ is isomorphic to $ H^1 (Y, \partial Y;M)$,
and that the bilinear form $\langle , \rangle_{A,B} $ coincides with the cup product in the $E_2$-term that we discussed in the proof with $m=3$.
Hence Proposition \ref{aa1133c} readily ensures the desired isomorphism.

Furthermore, we will compare our formulation of signature (Theorem \ref{akdsj}) with
results \cite{EN,Ozb}. 
According to Hopf theorem in group cohomology,
Endo and Nagami showed
$$ \mathrm{Sign } (E)-m+ m_{\rm ns} = \sum_{i=2}^{m} c_{ \M_g}(z_1 \cdots z_{i-1},z_i ), $$
for any Lefschetz fibration $E \ra S^2$ associated with any tuple $(z_1, \dots, z_m)$; or see \cite[Theorem 3]{Ozb}.
However, in general it is hard to compute signatures in $m$-times.
In conclusion, our formula in Theorem \ref{akdsj} is more applicable to computer program
than their formulae.
In computer experience, for $g \leq 9$, we
can easily compute signatures of many Lefschetz fibrations.
Actually, the signatures in Table \ref{G3} can be also computed in this way.

\vskip 1pc

\normalsize

Faculty of Mathematics, Kyushu University, 
744, Motooka, Nishi-ku, Fukuoka, 819-0395, Japan

\

E-mail address: {\tt nosaka@math.kyushu-u.ac.jp}

\end{document}